\documentclass[11pt, twoside]{article}

\usepackage[english]{babel}

\usepackage{amssymb}
\usepackage{mathrsfs}
\usepackage{amsmath}
\usepackage{amsthm}
\usepackage{amsfonts}
\usepackage{latexsym}
\usepackage{indentfirst}
\usepackage{color}
\usepackage{txfonts}
\usepackage{enumerate}

\usepackage[colorlinks=true,
linkcolor=blue,
citecolor=red,
urlcolor=magenta,
backref=page
]{hyperref}

\usepackage{txfonts}
\usepackage{anysize}

\allowdisplaybreaks

\newtheorem{theorem}{Theorem}[section]
\newtheorem{lemma}[theorem]{Lemma}

\newtheorem{proposition}[theorem]{Proposition}

\theoremstyle{definition}
\newtheorem{remark}[theorem]{Remark}

\newtheorem{definition}[theorem]{Definition}
\newcounter{assum}

\renewcommand{\appendix}{\par
\setcounter{section}{0}%
\setcounter{subsection}{0}%
\setcounter{subsubsection}{0}%
\gdef\thesection{\@Alph\c@section}%
\gdef\thesubsection{\@Alph\c@section.\@arabic\c@subsection}%
\gdef\theHsection{\@Alph\c@section.}%
\gdef\theHsubsection{\@Alph\c@section.\@arabic\c@subsection}%
\csname appendixmore\endcsname
}

\numberwithin{equation}{section}

\begin{document}
\title{\bf\Large
$\Gamma$-Convergence of Weak-Type Nonlocal  Functionals on Bounded
Domains\footnotetext{\hspace{-0.35cm} 2020 {\it
Mathematics Subject Classification}.
Primary 49J45; Secondary 46E35, 26B30, 26D10.
\endgraf {\it Key words and phrases.}
Gamma-convergence, nonlocal functional, Sobolev norm, bounded variation.
\endgraf This project is partially supported by the National
Natural Science Foundation of China (Grant Nos. 12431006, 12371093, and 12501118),
the Natural Science Foundation of Fujian Province (Grant No. 2026J008197),
the Beijing Natural Science Foundation (Grant No. 1262011),
the Fundamental Research Funds for the Central Universities
(Grant No. 2253200028), and Longyuan Young Talents of Gansu Province.}}
\author{Xiaosheng Lin, Dachun Yang, Sibei Yang, Wen Yuan
and Yangyang Zhang}
\date{\today}
\maketitle

\vspace{-0.7cm}

\begin{center}
\begin{minipage}{13cm}
{\small {\bf Abstract.}\quad
 Let $N\ge1$, $p\in[1,\infty)$, $\gamma\in(0,\infty)$, and
 $\Omega\subset\mathbb R^N$ be a bounded open interval
when $N=1$ or a bounded Lipschitz domain when $N\ge2$.
For any $\lambda\in(0,\infty)$ and any measurable function $u$,
consider the weak-type nonlocal functional
\begin{align*}
 G_{\lambda,p,\gamma}(u;\Omega)
 :=\lambda\iint_{\Omega\times\Omega}
 \mathbf 1_{\left\{(x,y)\in\Omega\times\Omega:\ x\neq y,\
 \frac{|u(x)-u(y)|^p}{|x-y|^{p+\gamma}}\geq\lambda\right\}}
 |x-y|^{\gamma-N}\,dx\,dy.
\end{align*}
In this article, we prove that, as $\lambda\to\infty$,
the family $G_{\lambda,p,\gamma}$ converges,
in the sense of $\Gamma$-convergence
in $L^p(\Omega)$, to the functional
\begin{align*}
 \Psi_{p,\gamma}^{\mathrm{cell}}(u;\Omega):=
\begin{cases}
 C_{N,p,\gamma}^{\mathrm{cell}}\displaystyle\int_\Omega|\nabla u|^p\,dx,
&p\in(1,\infty)\ \hbox{and}\ u\in W^{1,p}(\Omega),\\[2mm]
 C_{N,1,\gamma}^{\mathrm{cell}}|Du|(\Omega),
&p=1\ \hbox{and}\ u\in BV(\Omega),\\[1mm]
\infty,&\hbox{otherwise},
\end{cases}
\end{align*}
 where the positive constants $C_{N,p,\gamma}^{\mathrm{cell}}$
are independent of $\Omega$ and characterized by a cell formula. This
gives an affirmative answer to the problem posed by Brezis [Open Problem~9.3,
Rend. Lincei Mat. Appl. 2023].}
\end{minipage}
\end{center}

\vspace{0.1cm}

\tableofcontents
\section{Introduction}
Let $N\ge1$, $p\in [1,\infty)$, and $\Omega\subset \mathbb{R}^N$ be an open set.
Recall that the \emph{homogeneous Sobolev space} $\dot{W}^{1,p}(\Omega)$ is defined by setting
\begin{align*}
\dot{W}^{1,p}(\Omega):=\left\{f\in L^p_{\mathrm{loc}}(\Omega):\ |\nabla f|\in L^p(\Omega)\right\},
\end{align*}
where the \emph{notation $L^p_{\mathrm{loc}}(\Omega)$} denotes the set
of all locally $p$-integrable functions on
$\Omega$ and \emph{notation $\nabla f$} denotes the \emph{distributional gradient} of $f$.
Moreover, the \emph{nonhomogeneous Sobolev space} $W^{1,p}(\Omega)$ is defined as
$$W^{1,p}(\Omega):= L^p(\Omega)\cap \dot{W}^{1,p}(\Omega).$$
Meanwhile, denote by \emph{notation $\mathcal M(\Omega;\mathbb R^N)$}
the space of all finite $\mathbb R^N$-valued Radon
measures on $\Omega$. Then the homogeneous and the nonhomogeneous \emph{bounded-variation spaces} $\dot{BV}(\Omega)$
and $BV(\Omega)$ are respectively defined by setting
$$
\dot{BV}(\Omega):=\left\{u\in L^1_{\mathrm{loc}}(\Omega):Du\in\mathcal M(\Omega;\mathbb R^N)\right\}
$$
and
$$ BV(\Omega):=\dot{BV}(\Omega)\cap L^1(\Omega)$$
(see, for example, \cite{afp00}). Here, the \emph{notation $Du$}
denotes the distributional derivative of $u$,
and \emph{notation $|Du|$} denotes its total variation measure.

In the seminal articles \cite{bbm01, b02}, Bourgain, Brezis, and Mironescu
showed that,
for any $f\in W^{1,p}(\mathbb{R}^N)$,
\begin{align*}
\lim_{s\to 1^-}(1-s)\int_{\mathbb{R}^N}\int_{\mathbb{R}^N}\frac{|f(x)-f(y)|^p}{|x-y|^{N+sp}}\,dx\,dy
=\frac{K_{N,p}}{p}\|\,|\nabla f|\,\|_{L^p(\mathbb{R}^N)}^p,
\end{align*}
which is nowadays called the \emph{Bourgain--Brezis--Mironescu} (BBM) \emph{formula}. Here
\begin{align}\label{eq:gpgeo}
 K_{N,p}:=\int_{\mathbb S^{N-1}}|\omega\cdot e_N|^p\,d\mathcal H^{N-1}(\omega),
\end{align}
where $\mathbb S^{N-1}$ and $\mathcal H^{N-1}$ respectively
denote the \emph{unit sphere} in $\mathbb R^{N}$ and
the \emph{surface measure} on $\mathbb S^{N-1}$ and where
$e_N:=(0,\ldots,0,1)\in\mathbb S^{N-1}$.
Since then, the BBM formula and its various generalized versions have attracted a lot of attention
(see, for example, \cite{bn18,dgpyyz24,dlyyz23,dlyyz22,dm22,dm23,dssvy23,kmx05,zyy23}).
At the endpoint case $s=1$, Brezis et al. \cite{bvy21} proved that, for any
$f\in C_{\rm c}^\infty(\mathbb R^N)$ (the space of all infinitely
differentiable functions on $\mathbb R^N$ with compact support),
\begin{align*}
&\lim_{\lambda\to\infty}\lambda^p\mathcal L^{2N}
\left(\left\{(x,y)\in\mathbb R^N\times\mathbb R^N:
x\neq y,\ \frac{|f(x)-f(y)|}{|x-y|^{1+N/p}}\geq\lambda\right\}\right)
=\frac{K_{N,p}}{N}\|\,|\nabla f|\,\|_{L^p(\mathbb R^N)}^p,
\end{align*}
where $\mathcal L^{2N}$ denotes the Lebesgue measure on ${\mathbb R}^{2N}$.
For any given $\lambda\in(0,\infty)$, $p\in[1,\infty)$,
and $\gamma\in\mathbb{R}\backslash\{0\}$
and for any measurable function $u:\Omega\to\mathbb R$, let
\begin{align}\label{eq:lv}
F_{\lambda,p,\gamma}(u;\Omega):=\lambda^p\iint_{ \Omega\times\Omega}
\mathbf{1}_{\left\{(x,y)\in\Omega\times\Omega:\ x\neq y,\
\frac{|u(x)-u(y)|}{|x-y|^{1+\gamma/p}}>\lambda\right\}}
|x-y|^{\gamma-N}\,dx\,dy .
\end{align}
When $\Omega:=\mathbb R^N$, Brezis et al. \cite{bsvy24}
studied this family in a more general framework. Precisely, they obtained
sharp asymptotic estimates as $\lambda\to\infty$ when $\gamma\in(0,\infty)$ and as
$\lambda\downarrow0$ when $\gamma\in(-\infty,0)$ and, among other consequences, derived
characterizations of homogeneous Sobolev spaces.

These results led to a natural variational question.  In
\cite[Section~7C]{bsvy24}, Brezis et al. asked whether, for any given
$p\in[1,\infty)$ and $\gamma\in\mathbb{R}\backslash\{0\}$, there exists a positive constant
$c:=c(p,\gamma,N)$, depending only on $p$, $\gamma$, and $N$, such that the family
$F_{\lambda,p,\gamma}(\cdot;\mathbb{R}^N)$ converges, in the sense of $\Gamma$-convergence in $L^1_{\mathrm{loc}}(\mathbb R^N)$,
to the functional
\begin{align*}
\Psi_{\ast,c}(u):=\begin{cases}
 c\displaystyle\int_{\mathbb R^N}|\nabla u|^p\,dx,
&p\in(1,\infty)\ \hbox{and}\ u\in\dot{W}^{1,p}(\mathbb R^N),\\[2mm]
c|Du|(\mathbb R^N),
&p=1\ \hbox{and}\ u\in\dot{BV}(\mathbb R^N),\\[1mm]
\infty,&\hbox{otherwise},
\end{cases}
\end{align*}
as $\lambda\to\infty$ when $\gamma\in(0,\infty)$ and as $\lambda\downarrow0$ when
$\gamma\in(-\infty,0)$.  The case $p=1$ is especially delicate because the pointwise
limits on $\dot{W}^{1,1}(\mathbb R^N)$ and on $\dot{BV}(\mathbb R^N)$ may
differ. Furthermore, Brezis subsequently formulated in \cite[Open Problem~9.3]{b23} a
bounded-domain version of the case $\gamma=N$ of the above convergence of the family
$F_{\lambda,p,\gamma}(\cdot;\Omega)$ in \eqref{eq:lv}.
Let $\gamma\in (0,\infty)$
and $\Omega\subset\mathbb{R}^N$
be a bounded smooth domain.
For any $\lambda\in (0,\infty)$ and any  measurable function $u:\Omega\to\mathbb R$,
let
\begin{align}\label{eq:Ggamma}
G_{\lambda,p,\gamma}(u;\Omega):=\lambda\iint_{\Omega\times\Omega}
\mathbf 1_{\left\{(x,y)\in\Omega\times\Omega:\ x\neq y,\
\frac{|u(x)-u(y)|^p}{|x-y|^{p+\gamma}}\geq\lambda\right\}}
|x-y|^{\gamma-N}\,dx\,dy.
\end{align}
Brezis \cite[Open Problem~9.3]{b23} asked whether the family
$G_{\lambda,p,N}(\cdot;\Omega)$ converges as
$\lambda\to\infty$, in the sense of $\Gamma$-convergence
in $L^p(\Omega)$, to the functional
\begin{align}\label{e1.1}
\Psi_p(u;\Omega):=\begin{cases}
c_p\displaystyle\int_\Omega|\nabla u|^p\,dx,
&p\in(1,\infty)\ \hbox{and}\ u\in W^{1,p}(\Omega),\\[2mm]
c_1|Du|(\Omega),
&p=1\ \hbox{and}\ u\in BV(\Omega),\\[1mm]
\infty,&\hbox{otherwise}
\end{cases}
\end{align}
for some positive constants $c_p$ when $p\in(1,\infty)$ and $c_1$ when $p=1$.

Recall that, for the classical BBM energies, Ponce \cite{p04} showed that the corresponding
functionals $\Gamma$-converge with respect to the strong $L^p$ topology and that
the $\Gamma$-limit has the same constant as the pointwise BBM limit.
A different phenomenon occurs when
$\gamma:=-p$ in \eqref{eq:lv}: Nguyen \cite{n07,n11} proved that, as
$\lambda\downarrow0$, the family $F_{\lambda,p,-p}(\cdot;\mathbb{R}^N)$
$\Gamma$-converges in $L^p(\mathbb R^N)$ to a positive multiple of the
Sobolev energy when $p\in(1,\infty)$, or of the total variation when $p=1$,
and this positive multiplicative constant is strictly smaller than the pointwise-limit constant.
The positive multiplicative constant in this $\Gamma$-limit was given by an implicit cell
formula in \cite[formula~(1.3)]{n11}; its explicit value was later computed by
Antonucci et al. \cite[Theorem~1.1]{agmp20}.
Related nonlocal nonconvex approximations
and  their connections with image processing
 were further studied by Brezis and Nguyen
\cite{bn18,bn20ccm,bn20na}.
More recently, Nguyen
\cite[Theorem~29 and Remark~30]{n25} showed the following result:
if $p\in[1,\infty)$ and $\gamma\in(-1,0)$, then the family
$F_{\lambda,p,\gamma}(\cdot;\mathbb{R}^N)$ $\Gamma$-converges in
$L^p(\mathbb R^N)$, as $\lambda\downarrow0$, to the zero functional.
 For $\gamma\in(0,\infty)$, Gobbino and Picenni \cite{gp25}
established a positive lower bound for the $\Gamma$-liminf with respect to
$L^1_{\mathrm{loc}}(\mathbb R^N)$ convergence as $\lambda\to\infty$.
In the case $p=1$ and $\gamma\in(0,\infty)$, Lahti and Li \cite{ll26}
established the $\Gamma$-limit of
$F_{\lambda,1,\gamma}(\cdot;\Omega)$ with respect to area-strict convergence.
However, Brezis's bounded-domain $\Gamma$-convergence problem with respect to
the strong $L^p(\Omega)$ topology \cite[Open Problem~9.3]{b23} remains open.

The main result below gives an affirmative answer to Brezis's problem
\cite[Open Problem~9.3]{b23}, extends it to any $\gamma\in(0,\infty)$,
and identifies the positive constants in \eqref{e1.1} by means of the affine
cell constants defined below.

\begin{theorem}\label{thm:main}
Let $N\ge1$, $p\in [1,\infty)$, $\gamma\in(0,\infty)$, and $\Omega\subset\mathbb R^N$ be a bounded open interval
when $N=1$ or a bounded Lipschitz domain when $N\ge2$. Define the functional
$\Psi_{p,\gamma}^{\mathrm{cell}}:L^p(\Omega)\to[0,\infty]$ by setting, for any $u\in L^p(\Omega)$,
\begin{align*}
\Psi_{p,\gamma}^{\mathrm{cell}}(u;\Omega):=
\begin{cases}
C_{N,p,\gamma}^{\mathrm{cell}}\displaystyle\int_\Omega|\nabla u|^p\,dx,
&p\in(1,\infty)\ \hbox{and}\ u\in W^{1,p}(\Omega),\\[2mm]
C_{N,1,\gamma}^{\mathrm{cell}}|Du|(\Omega),
&p=1\ \hbox{and}\ u\in BV(\Omega),\\[1mm]
\infty,&\hbox{otherwise},
\end{cases}
\end{align*}
where the positive constant $C_{N,p,\gamma}^{\mathrm{cell}}$ is the same as in \eqref{eq:cellconst}.
Then the family $G_{\lambda,p,\gamma}(\cdot;\Omega)$ converges, in the sense of
$\Gamma$-convergence in $L^p(\Omega)$, to $\Psi_{p,\gamma}^{\mathrm{cell}}$ as $\lambda\to\infty.$
More precisely, let $\{\lambda_j\}_{j\in\mathbb N}\subset(0,\infty)$ be any sequence satisfying
$\lambda_j\to\infty$.  Then the following statements hold.
\begin{enumerate}[{\rm(i)}]
\item The $\Gamma$-liminf inequality: for any $u\in L^p(\Omega)$ and any
sequence $\{u_j\}_{j\in\mathbb N}\subset L^p(\Omega)$ satisfying
$u_j\to u$ in $L^p(\Omega)$,
\begin{align*}
\liminf_{j\to\infty}G_{\lambda_j,p,\gamma}(u_j;\Omega)\geq\Psi_{p,\gamma}^{\mathrm{cell}}(u;\Omega).
\end{align*}
\item The $\Gamma$-limsup inequality: for any $u\in L^p(\Omega)$, there
exists a sequence $\{u_j\}_{j\in\mathbb N}\subset L^p(\Omega)$ such that
$u_j\to u$ in $L^p(\Omega)$ and
\begin{align*}
\limsup_{j\to\infty}G_{\lambda_j,p,\gamma}(u_j;\Omega)
\leq\Psi_{p,\gamma}^{\mathrm{cell}}(u;\Omega).
\end{align*}
Such a sequence is called a \emph{recovery sequence} for $u$.
\end{enumerate}
\end{theorem}

\begin{remark}
The pointwise limit of the BSVY-type functionals \eqref{eq:lv}
need not exist for $BV$ functions. More precisely, Lahti
\cite[Example 3.22]{la25} showed that, for any given $\gamma\in (0,\infty)$, there exists
$u\in BV(0,1)$ such that the limit
\begin{align*}
\lim_{\lambda\to\infty}
F_{\lambda,1,\gamma}(u;(0,1))
\end{align*}
does not exist. However, Theorem \ref{thm:main}, together with the
relationships between $F_{\lambda,1,\gamma}$ and $G_{\lambda,1,\gamma}$,
shows that, for every fixed
$\gamma\in(0,\infty)$, the family $F_{\lambda,1,\gamma}(\cdot;\Omega)$
$\Gamma$-converges in $L^1(\Omega)$, as $\lambda\to\infty$, to the functional $\Psi_{1,\gamma}^{\mathrm{cell}}(\cdot;\Omega)$.
Thus, although the pointwise limit $F_{\lambda,1,\gamma}(u;\Omega)$ may fail to
exist for a fixed $u\in BV(\Omega)$, the corresponding $\Gamma$-limit exists
and satisfies
\begin{align*}
\left(
\Gamma\mathop{\!-\!\lim}_{\lambda\to\infty}
F_{\lambda,1,\gamma}(\cdot;\Omega)
\right)(u)
=
C_{N,1,\gamma}^{\mathrm{cell}}|Du|(\Omega).
\end{align*}
For further results on the pointwise asymptotic behavior of
these functionals on $BV$ spaces, we refer to Lahti \cite{la25}
and Picenni \cite{pic24}.
\end{remark}

The proof of Theorem \ref{thm:main} has two main ingredients,
which allow us to identify the exact constant in the $\Gamma$-limit
and to construct a matching recovery sequence.
For the \emph{lower bound}, different from the dyadic discretization 
and the slicing arguments of Gobbino and Picenni \cite{gp25}, 
we compare the energy on small cubes directly with the affine cell constant.
Moreover, when $p\in(1,\infty)$, we show that the normalized blow-ups of a 
Sobolev function at almost every point with nonzero gradient converge to an affine function.
This gives the required lower bound on sufficiently small cubes, 
and a Morse-type covering argument then yields the global estimate.
The main difficulty for the lower bound appears when $p=1$. 
In this case, by using the fine structure of BV functions, we find that 
normalized blow-ups need not be affine and may converge to general 
one-dimensional nondecreasing BV functions.
We therefore prove that the cell energy of every such function is bounded from below by its total variation times the same affine cell constant $C_{N,1,\gamma}^{\mathrm{cell}}$.
For the \emph{upper bound}, the nonlocal nature of the functional
prevents us from directly gluing scaled cell recovery sequences because interactions between different cubes may create additional energy.
To deal with this problem, motivated by Nguyen \cite{n11},
we modify near-minimizing sequences for the affine cell 
constant so that they agree with the affine function near the boundary of the cube,
and then place scaled copies of these modified sequences 
into a continuous piecewise affine approximation.
Combining the lower and the upper bounds yields the desired full $\Gamma$-limit 
with the constant $C_{N,p,\gamma}^{\mathrm{cell}}$.

The remainder of this article is organized as follows.

In Section \ref{sec:tools}, we introduce the cell constant
$C_{N,p,\gamma}^{\mathrm{cell}}$ and establish its basic properties.
In Section \ref{sec:plimsup}, we show the $\Gamma$-limsup inequality.
Specifically, in Subsection \ref{subsec:bdrymod}, we modify
the sequences realizing the cell constant so that they agree with affine functions
near the boundary. In Subsection \ref{subsec:parec}, we use
these modified sequences to construct recovery sequences for continuous
piecewise affine functions. Finally, in Subsection \ref{proofsup}, we obtain
the $\Gamma$-limsup inequality for general Sobolev and BV functions by
approximation. In Section \ref{sec:liminf}, we prove the $\Gamma$-liminf inequality.
Specifically, in Subsection \ref{proofp}, we treat the case $p\in(1,\infty)$.
 In Subsection
\ref{subsec:bvlim}, we study one-dimensional limits of rescaled BV
functions.
In
Subsection \ref{subsec:monobv}, we establish a lower bound for nondecreasing BV
functions of one variable.
Finally, in Subsection \ref{proofone}, we use these results to show
the $\Gamma$-liminf inequality for $p=1$.

At the end of this section, we make some notational conventions. Let
${\mathbb N}:=\{1,2,\ldots\}$ and ${\mathbb Z}_+:={\mathbb N}\cup\{0\}$.
For any $\beta:=(\beta_1,\ldots,\beta_N)\in\mathbb{Z}_+^N:=(\mathbb{Z}_+)^N$,
let $|\beta|:=\beta_1+\cdots+\beta_N$
and, for any $x:=(x_1,\ldots,x_N)$, let
$x^\beta:=x_1^{\beta_1}\cdots x_N^{\beta_N}$ and
$\partial^\beta:=\left(\frac{\partial}{\partial x_1}\right)^{\beta_1}\cdots
\left(\frac{\partial}{\partial x_N}\right)^{\beta_N}$.
We always denote by $C$ a positive constant
which is independent of the main parameters involved,
but it may vary from line to line.
We also use $C_{\alpha,\beta,\ldots}$ to denote a
positive constant depending on the indicated parameters
$\alpha$, $\beta,\ldots$.
The notation $f\lesssim g$ means that $f\leq Cg$.
If $f\lesssim g$ and $g\lesssim f$,
we then write $f\sim g$. If $f\leq Cg$ and $g=h$
or $g\leq h$, we then write $f\lesssim g=h$ or $f\lesssim g\leq h$.
For any $x\in{\mathbb R}^N$ and $r\in(0,\infty)$,
we denote by
$B(x,r):=\{y\in{\mathbb R}^N:\ |y-x|<r\}$
the ball with center $x$ and radius $r$.
We use $\mathbf{1}_E$ to denote the
characteristic function of a measurable set $E\subset{\mathbb R}^N$
and $\mathbf{0}$ to denote the \emph{origin} of ${\mathbb R}^N$.
We denote by $\mathcal L^N$ the $N$-dimensional Lebesgue measure
and by $\mathcal H^k$ the $k$-dimensional Hausdorff measure.
For any $x\in\mathbb{R}^N$ and any nonempty set $E\subset\mathbb{R}^N$, let
$\mathrm{dist}(x,E):=\inf\{|x-y|:y\in E\}$. If
$\mu$ is a Radon measure and $A$ is a Borel set, then
$\mu\llcorner A$ denotes the restriction of $\mu$ to $A$, namely
\begin{align*}
 (\mu\llcorner A)(E):=\mu(E\cap A)
\end{align*}
for every Borel set $E$. If $\mu$ is a positive Radon measure and
$f\in L^1(\mu;\mathbb R^m)$, then $f\mu$ denotes the
$\mathbb R^m$-valued Radon measure defined by
\begin{align*}
 (f\mu)(E):=\int_E f\,d\mu.
\end{align*}
For a finite signed or vector-valued Radon measure $\sigma$, we
denote its total variation measure by $|\sigma|$.
For any given $N\in\mathbb N$, let
$$\mathcal O(N):=\left\{R\ \text{is an}\ N\times N \ \text{matrix}: R^TR=\mathrm{I}_N\right\},
$$
where $R^T$ denotes the transpose matrix of $R$
and $\mathrm{I}_N$ denotes the $N\times N$ identity matrix.
Throughout this article, we always let
\begin{align}\label{eq:ya}
I_0:=\left(-\frac{1}{2},\frac{1}{2}\right)\ \ \text{and}\ \ Q_0:=I_0^N=\left(-\frac{1}{2},\frac{1}{2}\right)^N.
\end{align}
Moreover, for any $x:=(x_1,\ldots,x_N)\in\mathbb{R}^N,$ define
\begin{equation}\label{e1.2}
\ell(x):=x_N.
\end{equation}
For any measurable set $E\subset\mathbb{R}^N$
and any $f\in L^1(E)$, we denote the integral $\int_{E}f(x)\,dx$ simply by
$\int_{E}f\,dx$ and, when $|E|\in(0,\infty)$, we always use the following notation
$$
\fint_Ef\,dx:=\frac{1}{|E|}\int_{E}f(x)\,dx.
$$
Finally, in all proofs we consistently retain the notation
introduced in the original theorem (or related statement).

\section{Preliminary}\label{sec:tools}
In this section, we introduce the cell constant
$C_{N,p,\gamma}^{\mathrm{cell}}$ and establish its basic properties.

\begin{definition}
For any given $p\in[1,\infty)$, $\gamma\in(0,\infty)$, and $u\in L^p(Q_0)$, define
\begin{align}\label{eq:mp}
m_{p,\gamma}(u):=\inf\left\{\liminf_{j\to\infty}G_{\lambda_j,p,\gamma}(u_j;Q_0):
\lambda_j\to\infty,\quad u_j\to u\ \hbox{in }L^p(Q_0)\right\},
\end{align}
where $G_{\lambda_j,p,\gamma}$ is as in \eqref{eq:Ggamma} and $Q_0$ is as in  \eqref{eq:ya}.
Moreover, let $C_{N,p,\gamma}^{\mathrm{cell}}:=m_{p,\gamma}(\ell)$ with $\ell$ being as in
\eqref{e1.2}; i.e.,
\begin{align}\label{eq:cellconst}
C_{N,p,\gamma}^{\mathrm{cell}}=\inf\left\{
\liminf_{j\to\infty}G_{\lambda_j,p,\gamma}(u_j;Q_0):
\lambda_j\to\infty,\quad u_j\to\ell\ \hbox{in }L^p(Q_0)\right\}.
\end{align}
\end{definition}

It is easy to verify that $m_{p,\gamma}$ has the following scaling properties.
\begin{lemma}\label{lem:scaling}
Let $N\in\mathbb{N}$, $p\in[1,\infty)$, $\gamma\in(0,\infty)$, $U\subset\mathbb R^N$ be an open set, and $w:U\to\mathbb R$ be a measurable function. Let $x_0\in\mathbb R^N$, $R\in\mathcal O(N)$, $r\in(0,\infty)$, $a\in\mathbb R\setminus\{0\}$, and $b\in\mathbb R$. Define the function $u:x_0+rRU\to\mathbb R$ by setting,
for any $z\in U$,
\begin{align*}
u(x_0+rRz):=arw(z)+b.
\end{align*}
Then
\begin{align*}
G_{\lambda,p,\gamma}(u;x_0+rRU)=|a|^p r^NG_{\frac{r^\gamma}{|a|^p}\lambda,p,\gamma}(w;U).
\end{align*}
Consequently,
\begin{align}\label{eq:mpscale}
 m_{p,\gamma}(aw+b)=|a|^p m_{p,\gamma}(w).
\end{align}
\end{lemma}

The following lemma proves that $m_{p,\gamma}$ is lower semicontinuous.
\begin{lemma}\label{lem:mplsc}
Let $N\in\mathbb{N}$, $p\in [1,\infty)$, and $\gamma\in(0,\infty).$ The map $m_{p,\gamma}:L^p(Q_0)\to[0,\infty]$ is
lower semicontinuous; i.e., for any sequence $\{w_n\}_{n\in\mathbb{N}}\subset
L^p(Q_0)$ converging to $w$ in $L^p(Q_0)$,
\begin{align}\label{genggao}
m_{p,\gamma}(w)\leq\liminf_{n\to\infty}m_{p,\gamma}(w_n).
\end{align}
\end{lemma}

\begin{proof}
Let $L:=\liminf_{n\to\infty}m_{p,\gamma}(w_n)\in[0,\infty].$ If $L=\infty$, then
the desired conclusion holds automatically. Assume $L<\infty$.
Passing to a subsequence and relabelling if necessary, we may assume that $\lim_{n\to\infty}m_{p,\gamma}(w_n)= L$ and $m_{p,\gamma}(w_n)<\infty$ for any $n\in\mathbb{N}.$
For any given $n\in\mathbb{N}$, choose an admissible sequence $\{(\lambda_{n,k},v_{n,k})\}_{k\in\mathbb{N}}$ in \eqref{eq:mp} for $w_n$
such that
\begin{align*}
\liminf_{k\to\infty}G_{\lambda_{n,k},p,\gamma}(v_{n,k};Q_0)\leq m_{p,\gamma}(w_n)+\frac{1}{2n}.
\end{align*}
Using this, the convergence of $\lambda_{n,k}$ and $v_{n,k}$, and the definition
of the limit inferior, we conclude that, for any given  $n\in\mathbb N$,
there exists $k(n)\in\mathbb{N}$ sufficiently large such that
$\lambda_{n,k(n)}\geq n,$ $\|v_{n,k(n)}-w_n\|_{L^p(Q_0)}\leq\frac1n,$
and
\begin{align*}
G_{\lambda_{n,k(n)},p,\gamma}(v_{n,k(n)};Q_0)
\leq m_{p,\gamma}(w_n)+\frac1n.
\end{align*}
This implies that \eqref{genggao} holds, which completes the proof of Lemma
\ref{lem:mplsc}.
\end{proof}

The following lemma is a part of  \cite[Theorem~1.1]{gp25}.
\begin{lemma}\label{lem:gp}
Let $N\in\mathbb{N}$, $p\in[1,\infty)$, $\gamma\in(0,\infty)$, and $U\subset\mathbb R^N$ be an
open set. Then, for any $u\in L^1_{\mathrm{loc}}(U)$ and any sequences $\{\lambda_j\}_{j\in\mathbb{N}}$ and
$\{u_j\}_{j\in\mathbb{N}}$ satisfying $\lim_{j\to\infty}\lambda_j=\infty$ and
$\lim_{j\to\infty}u_j=u$ in $L^1_{\mathrm{loc}}(U),$
\begin{align*}
\liminf_{j\to\infty}G_{\lambda_j,p,\gamma}(u_j;U)\geq K_{N,p}\frac{\log 2}{2^{\gamma+1}-1}
\begin{cases}
\displaystyle\int_U|\nabla u|^p\,dx,&p\in(1,\infty)\hbox{ and }u\in\dot W^{1,p}(U),\\[2mm]
|Du|(U),&p=1\hbox{ and }u\in\dot{BV}(U),\\[1mm]
\infty,&\hbox{otherwise},
\end{cases}
\end{align*}
where $K_{N,p}$ is as in \eqref{eq:gpgeo}.
\end{lemma}

Indeed, this formulation follows from \cite[Theorem~1.1]{gp25} by applying
that theorem with the parameter $\lambda_j^{1/p}$. The non-strict inequality
in \eqref{eq:Ggamma} only enlarges the corresponding superlevel set.

By Lemma \ref{lem:gp}, we have the following conclusion.
\begin{lemma}\label{lem:cellbds}
Let $p\in[1,\infty)$ and $\gamma\in(0,\infty)$. Then the following statements hold.
\begin{itemize}
\item[$\mathrm{(i)}$]
$\lim_{\lambda\to\infty}G_{\lambda,p,\gamma}(\ell;Q_0)=K_{N,p}/\gamma$,
where $K_{N,p}$ is as in \eqref{eq:gpgeo}.
Consequently, $C_{N,p,\gamma}^{\mathrm{cell}}\leq K_{N,p}/{\gamma}.$

\item [$\mathrm{(ii)}$] There exist sequences $\{\lambda_j\}_{j\in\mathbb{N} }$
and $\{v_j\}_{j\in\mathbb{N}}$ such that $\lim_{j\to\infty}\lambda_j=\infty$, $\lim_{j\to\infty}v_j=\ell$ in $L^p(Q_0)$, and
\begin{align}\label{eq:nearcell}
\lim_{j\to\infty}G_{\lambda_j,p,\gamma}(v_j;Q_0)=C_{N,p,\gamma}^{\mathrm{cell}}.
\end{align}
Moreover,
\begin{align}\label{eq:cellbds}
K_{N,p}\frac{\log2}{2^{\gamma+1}-1}\leq C_{N,p,\gamma}^{\mathrm{cell}}.
\end{align}
\end{itemize}
\end{lemma}

\begin{proof}
By Fubini's theorem, we find that, for any $\lambda\in(0,\infty),$
\begin{align*}
G_{\lambda,p,\gamma}(\ell;Q_0)&=\lambda\int_{\mathbb R^N}|Q_0\cap(Q_0-h)|
\mathbf 1_{\left\{h\in\mathbb R^N:\ h\neq\mathbf 0,\ \frac{|h_N|^p}{|h|^{p+\gamma}}\geq\lambda\right\}}
|h|^{\gamma-N}\,dh\\
&=\int_{\mathbb R^N}\left|Q_0\cap\left(Q_0-\lambda^{-1/\gamma}z\right)\right|
\mathbf 1_{\left\{z\in\mathbb R^N:\ z\neq\mathbf 0,\ \frac{|z_N|^p}{|z|^{p+\gamma}}\geq1\right\}}
|z|^{\gamma-N}\,dz.
\end{align*}
This, combined with the dominated convergence theorem, implies that
\begin{align*}
\lim_{\lambda\to\infty}G_{\lambda,p,\gamma}(\ell;Q_0)
&=\int_{\mathbb R^N}\mathbf 1_{\left\{z\in\mathbb R^N:\ z\neq\mathbf 0,\ \frac{|z_N|^p}{|z|^{p+\gamma}}\geq1\right\}}
|z|^{\gamma-N}\,dz
=\int_{\mathbb S^{N-1}}\int_0^{|\omega\cdot e_N|^{p/\gamma}}r^{\gamma-1}\,dr\,
d\mathcal H^{N-1}(\omega)\\
&=\frac1\gamma\int_{\mathbb S^{N-1}}|\omega\cdot e_N|^p\,d\mathcal H^{N-1}(\omega)
=\frac{K_{N,p}}{\gamma}.
\end{align*}
This shows (i).

Moreover, \eqref{eq:nearcell} follows directly from the standard
diagonal argument and the definition of $C_{N,p,\gamma}^{\mathrm{cell}}$, and  \eqref{eq:cellbds} follows directly
from  \eqref{eq:nearcell} and Lemma \ref{lem:gp}. This completes the proof of
(ii) and hence Lemma \ref{lem:cellbds}.
\end{proof}

\section{The Gamma-Limsup Inequality\label{sec:plimsup}}

In this section, we prove the Gamma-limsup inequality stated  in Theorem \ref{thm:main}.
More precisely, we have the  following  conclusions.

\begin{theorem}\label{thm:plimsup}
Let $N\in\mathbb N$, $p\in [1,\infty)$, $\gamma\in(0,\infty)$, and $\Omega\subset\mathbb R^N$ be a bounded open interval
when $N=1$ or a bounded Lipschitz domain when $N\ge2$. Let $u\in W^{1,p}(\Omega)$ and
$\{\lambda_j\}_{j\in\mathbb N}\subset(0,\infty)$ satisfy $\lim_{j\to\infty}\lambda_j=\infty$.
Then there exists $\{u_j\}_{j\in\mathbb N}\subset L^p(\Omega)$ satisfying $u_j\to u$ in
$L^p(\Omega)$ and
\begin{align}\label{eq:jiahua}
\limsup_{j\to\infty}G_{\lambda_j,p,\gamma}(u_j;\Omega)
\leq C_{N,p,\gamma}^{\mathrm{cell}}\int_\Omega|\nabla u|^p\,dx,
\end{align}
where $C_{N,p,\gamma}^{\mathrm{cell}}$ is as in \eqref{eq:cellconst}.
\end{theorem}
\begin{theorem}\label{thm:onelimsup}
Let $\gamma\in(0,\infty)$ and $\Omega\subset \mathbb{R}^N $ be as in Theorem \ref{thm:plimsup},
$u\in BV(\Omega)$, $\{\lambda_j\}_{j\in\mathbb N}\subset(0,\infty)$ satisfy
$\lim_{j\to\infty}\lambda_j=\infty$.  Then there exists
$\{u_j\}_{j\in\mathbb N}\subset L^1(\Omega)$ satisfying $u_j\to u$ in $L^1(\Omega)$ and
\begin{align*}
\limsup_{j\to\infty}G_{\lambda_j,1,\gamma}(u_j;\Omega)\leq C_{N,1,\gamma}^{\mathrm{cell}}|Du|(\Omega),
\end{align*}
where $C_{N,1,\gamma}^{\mathrm{cell}}$ is as in \eqref{eq:cellconst}.
\end{theorem}
The proofs of Theorems \ref{thm:plimsup} and \ref{thm:onelimsup} are given in Subsection \ref{proofsup}.

\subsection{Boundary Modification of Sequences Realizing the Affine Cell Constant}
\label{subsec:bdrymod}
In this  subsection, we use Lemma \ref{lem:cellbds} to show the following  conclusion, which
plays a pivotal role in the proof of Theorem \ref{thm:plimsup}.
\begin{proposition}\label{prop:prepair}
Let $p\in[1,\infty)$ and $\gamma\in(0,\infty)$. Then  there exist sequences $\{\Lambda_j\}_{j\in\mathbb N}\subset(0,\infty)$,
$\{V_j\}_{j\in\mathbb N}\subset L^\infty(Q_0)$, and $\{\rho_j\}_{j\in\mathbb N}\subset(0,\infty)$
satisfying the following properties:
\begin{enumerate}[{\rm(i)}]
\item $\lim_{j\to\infty}\Lambda_j=\infty$, $\lim_{j\to\infty}V_j=\ell$ in $L^p(Q_0)$,
$\lim_{j\to\infty}\rho_j=0$, and
\begin{align}\label{eq:pah1}
\lim_{j\to\infty}G_{\Lambda_j,p,\gamma}(V_j;Q_0)=C_{N,p,\gamma}^{\mathrm{cell}}.
\end{align}
\item  for any $j\in\mathbb{N}$,
\begin{align}\label{eq:pah2}
V_j=\ell\ \ \text{on}\ \
\{z\in Q_0:\operatorname{dist}(z,\partial Q_0)<\rho_j\}.
\end{align}
\item it holds that
\begin{align}\label{eq:pah3}
\lim_{j\to\infty}\Lambda_j\rho_j^\gamma=\infty\ \ \text{and}\ \
\lim_{j\to\infty}\frac{\|V_j-\ell\|_{L^\infty(Q_0)}^p}
{\Lambda_j\rho_j^{p+\gamma}}=0.
\end{align}
\end{enumerate}
\end{proposition}
To prove Proposition \ref{prop:prepair}, we need the following basic results.
Let $p\in[1,\infty)$ and $\gamma\in(0,\infty)$. For any given  measurable sets $E\subset D\subset\mathbb R^N$ and any
measurable function $u:D\to\mathbb R$, let
\begin{align*}
\mathcal G_{\lambda,p,\gamma}(u;E,D):=\lambda\int_E\int_D
\mathbf 1_{\left\{(x,y)\in E\times D:\ x\neq y,\ \frac{|u(x)-u(y)|^p}{|x-y|^{p+\gamma}}\geq\lambda\right\}}
|x-y|^{\gamma-N}\,dy\,dx.
\end{align*}

\begin{lemma}\label{lem:pcutoff}
Let $N\in\mathbb N$, $p\in [1,\infty)$, $\gamma\in(0,\infty)$, and $\lambda\in (0,\infty).$
\begin{itemize}
\item [$\mathrm{(i)}$] For any given  measurable sets $E\subset D\subset\mathbb R^N$,
\begin{align*}
\mathcal G_{\lambda,p,\gamma}(\ell;E,D)\leq\frac{\omega_N}{\gamma}|E|,
\end{align*}
where $\omega_N:=\mathcal H^{N-1}(\partial B(\mathbf{0},1)).$ Consequently, $G_{\lambda,p,\gamma}(\ell;E)\leq\frac{\omega_N}{\gamma}|E|$.
\item  [$\mathrm{(ii)}$]Let $U\subset\mathbb R^N$ be a bounded open set,
$u,v\in L^p(U)$, and $\eta:U\to[0,1]$ be Lipschitz.
Then, for any  $\alpha\in (0,1)$,
\begin{align}\label{eq:pcutoff}
G_{\lambda,p,\gamma}(w;U)\leq\frac{G_{\alpha^p\lambda,p,\gamma}(u;U)}{\alpha^p}
+\frac{\mathcal G_{\alpha^p\lambda,p,\gamma}(v;E_\eta,U)}{\alpha^p}
+\frac{\omega_NL_\eta^p\|u-v\|_{L^p(U)}^p}{\gamma(1-\alpha)^p},
\end{align}
where $w:=\eta u+(1-\eta)v,$ $E_\eta:=\{x\in U:\eta(x)<1\}$, and
\begin{align*}
L_\eta:=\sup_{\substack{x,y\in U\\x\neq y}}\frac{|\eta(x)-\eta(y)|}{|x-y|}<\infty.
\end{align*}
\end{itemize}
\end{lemma}

\begin{proof}
Note that, for any $x,y\in\mathbb{R}^N,$ $|\ell(x)-\ell(y)|\leq|x-y|.$ Thus,
\begin{align*}
\mathcal G_{\lambda,p,\gamma}(\ell;E,D)\leq\lambda \int_{E}
\int_{D\cap B(x, \lambda^{-\frac{1}{\gamma}})}|x-y|^{\gamma-N}\,dy\,dx
\leq\frac{\omega_N}{\gamma}|E|.
\end{align*}
This shows (i).

Moreover, observe that, for any $x,y\in U$,
\begin{align*}
w(x)-w(y)=A(x,y)+B(x,y),
\end{align*}
where
\begin{align*}
 A(x,y):=\eta(x)[u(x)-u(y)]+[1-\eta(x)][v(x)-v(y)]
\end{align*}
and
\begin{align*}
 B(x,y):=[\eta(x)-\eta(y)][u(y)-v(y)].
\end{align*}
Set
$$E_w:=\left\{(x,y)\in U\times U:\ x\neq y,\ \frac{|w(x)-w(y)|}{|x-y|^{\frac{p+\gamma}{p}}}\geq \lambda^{\frac{1}{p}}\right\},$$
$$E_{A}:=\left\{(x,y)\in U\times U:\ x\neq y,\ \frac{|A(x,y)|}{|x-y|^{\frac{p+\gamma}{p}}}\geq \alpha \lambda^{\frac{1}{p}}\right\},$$
and
$$E_{B}:=\left\{(x,y)\in U\times U:\ x\neq y,\ \frac{|B(x,y)|}{|x-y|^{\frac{p+\gamma}{p}}}\geq (1-\alpha) \lambda^{\frac{1}{p}}\right\}.$$
Then
\begin{align}\label{eq:joining2}
E_w\subset E_A\cup E_B.
\end{align}
By convexity of the function $f(t):=|t|^p$ in $\mathbb{R}$, we find that, for any $x,y\in U,$
\begin{align*}
|A(x,y)|^p\leq\eta(x)|u(x)-u(y)|^p+[1-\eta(x)]|v(x)-v(y)|^p.
\end{align*}
This further implies that
\begin{align}\label{eq:joining}
E_A\subset E_{A,1}\cup E_{A,2},
\end{align}
where
\begin{align*}
E_{A,1}:=\left\{(x,y)\in U\times U:\ x\neq y,\ \frac{|u(x)-u(y)|^p}{|x-y|^{p+\gamma}}\geq\alpha^p\lambda\right\}
\end{align*}
and
\begin{align*}
 E_{A,2}:=\left\{(x,y)\in E_\eta\times U:\ x\neq y,\ \frac{|v(x)-v(y)|^p}{|x-y|^{p+\gamma}}\geq\alpha^p\lambda\right\}.
\end{align*}
From the definition of $L_\eta$, we deduce that, for any $(x,y)\in E_B$,
\begin{align*}
(1-\alpha)^p\lambda|x-y|^{p+\gamma}\leq|\eta(x)-\eta(y)|^p|u(y)-v(y)|^p
\leq L_\eta^p|x-y|^p|u(y)-v(y)|^p
\end{align*}
and hence
\begin{align*}
|x-y|^\gamma\leq\frac{L_\eta^p}{(1-\alpha)^p\lambda}|u(y)-v(y)|^p,
\end{align*}
which implies that
\begin{align*}
\lambda\iint_{E_B}|x-y|^{\gamma-N}\,dy\,dx
\leq\frac{\omega_N L_\eta^p\|u-v\|^p_{L^p(U)}}{\gamma(1-\alpha)^p}.
\end{align*}
By this, \eqref{eq:joining}, and \eqref{eq:joining2}, we conclude that \eqref{eq:pcutoff} holds.
This completes the proof of Lemma \ref{lem:pcutoff}.
\end{proof}

Now, we prove
Proposition \ref{prop:prepair}.
\begin{proof}[Proof of Proposition \ref{prop:prepair}]
Let $\{\widetilde\lambda_n\}_{n\in\mathbb N}\subset(0,\infty)$ and $\{v_n\}_{n\in\mathbb N}
\subset L^p(Q_0)$ be the same as in Lemma \ref{lem:cellbds}(ii). Fix $j\in\mathbb{N}$
and define the truncation function $T_j$ by setting, for any $t\in\mathbb R$,
\begin{align*}
T_j(t)=\begin{cases}
-j, & t\in(-\infty,-j),\\
t,  & t\in[-j,j],\\
j,  & t\in(j,\infty).
\end{cases}
\end{align*}
It is easy to check that  $T_j$ is $1$-Lipschitz; i.e., for any $t_1,t_2\in\mathbb{R},$
$|T_j(t_1)-T_j(t_2)|\leq |t_1-t_2|.$ From this and the fact that $T_j(\ell)=\ell$ on $Q_0$,
we infer that, for any $n\in\mathbb{N},$
\begin{align*}
|T_j(v_n)-\ell|=|T_j(v_n)-T_j(\ell)|\leq|v_n-\ell|.
\end{align*}
This further implies that
\begin{align*}
\lim_{n\to\infty}T_j(v_n)=\ell\ \ \text{in}\ \ L^p(Q_0).
\end{align*}
Moreover, for any $\lambda\in (0,\infty)$,
\begin{align*}
G_{\lambda,p,\gamma}\left(T_j(v_n);Q_0\right)\leq G_{\lambda,p,\gamma}\left(v_n;Q_0\right).
\end{align*}
Define the $1$-Lipschitz function $\theta:\mathbb R\to[0,1]$ by setting, for any $t\in\mathbb R$,
\begin{align*}
\theta(t):=
\begin{cases}
0,&t\in(-\infty,3],\\
t-3,&t\in(3,4),\\
1,&t\in[4,\infty),
\end{cases}
\end{align*}
and, for any $z\in Q_0$, let
\begin{align*}
\eta_j(z):=\theta\left(\frac{\operatorname{dist}(z,\partial Q_0)}{\rho_j}\right)=\begin{cases}
0,&\operatorname{dist}(z,\partial Q_0)\leq 3\rho_j,\\[1mm]
\displaystyle\frac{\operatorname{dist}(z,\partial Q_0)}{\rho_j}-3,
&3\rho_j<\operatorname{dist}(z,\partial Q_0)<4\rho_j,\\[2mm]
1,&\operatorname{dist}(z,\partial Q_0)\geq 4\rho_j.
\end{cases}
\end{align*}
where $\rho_j:=\frac{1}{100j}$ for any $j\in\mathbb{N}.$
Then, for any $z,w\in Q_0$,
\begin{align*}
|\eta_j(z)-\eta_j(w)|&=\left|\theta\left(
\frac{\operatorname{dist}(z,\partial Q_0)}{\rho_j}\right)-\theta\left(
\frac{\operatorname{dist}(w,\partial Q_0)}{\rho_j}\right)\right|\\
&\leq\left|\frac{\operatorname{dist}(z,\partial Q_0)}{\rho_j}-
\frac{\operatorname{dist}(w,\partial Q_0)}{\rho_j}\right|\leq \frac{|z-w|}{\rho_j}.
\end{align*}
Thus, $\eta_j$ is a Lipschitz function on  $Q_0$ with Lipschitz constant at most $\frac{1}{\rho_j}$.
For any $j\in\mathbb N$, let
\begin{align*}
E_j:=\{z\in Q_0:\eta_j(z)<1\}.
\end{align*}
Then $E_j$ is contained in $\{z\in Q_0:\ \operatorname{dist}(z,\partial Q_0)<4\rho_j\}.$ Thus,
$|E_j|\leq C_N\rho_j$.

Now, for any $j,n\in\mathbb{N}$, let
\begin{align*}
 V_{j,n}:=\eta_jT_j(v_n)+(1-\eta_j)\ell\ \ \text{and}\ \ \lambda_{j,n}:=\frac{\widetilde\lambda_n}{\alpha_j^p},
\end{align*}
where $\{\alpha_j\}_{j\in\mathbb{N}}\subset (0,1)$ satisfies $\lim_{j\to\infty}\alpha_j=1.$
Applying Lemma \ref{lem:pcutoff}, we find that, for any $j,n\in\mathbb{N}$,
\begin{align}\label{eq:repairbd}
G_{\lambda_{j,n},p,\gamma}(V_{j,n};Q_0)&\leq
\frac{G_{\alpha_j^p\lambda_{j,n},p,\gamma}(T_j(v_n);Q_0)}{\alpha_j^p}
+\frac{\mathcal G_{\alpha_j^p\lambda_{j,n},p,\gamma}(\ell;E_j,Q_0)}{\alpha_j^p}
+\frac{\omega_N\|T_j(v_n)-\ell\|_{L^p(Q_0)}^p}{\gamma(1-\alpha_j)^p\rho_j^p}\notag\\
&\leq\frac{G_{\widetilde\lambda_n,p,\gamma}(v_n;Q_0)}{\alpha_j^p}+\frac{C_{N,\gamma}\rho_j}{\alpha_j^p}
+\frac{\omega_N\|T_j(v_n)-\ell\|_{L^p(Q_0)}^p}{\gamma(1-\alpha_j)^p\rho_j^p}.
\end{align}
Note that, for any given $j\in\mathbb{N}$, $\lim_{n\to\infty}V_{j,n}=\ell$ in $L^p(Q_0)$,
$\lim_{n\to\infty}\lambda_{j,n}=\infty$, and
\begin{align*}
\lim_{n\to\infty}\frac{\|V_{j,n}-\ell\|^p_{L^\infty(Q_0)}}{\lambda_{j,n}\rho_{j}^{p+\gamma}}\leq
\lim_{n\to\infty}\frac{(j+\frac{1}{2})^p}{\lambda_{j,n}\rho_{j}^{p+\gamma}}=0.
\end{align*}
This, combined with \eqref{eq:repairbd},  implies that, for any  given $j\in\mathbb{N}$, there exists
$n(j)\in\mathbb{N}$ such that
\begin{align*}
\lambda_{j,n(j)}\rho_j^\gamma\geq j,\ \
\frac{\|V_{j,n(j)}-\ell\|_{L^\infty(Q_0)}^p}{\lambda_{j,n(j)}\rho_j^{p+\gamma}}\leq \frac{1}{j},
\ \ \|V_{j,n(j)}-\ell\|_{L^p(Q_0)}\leq\frac1j,
\end{align*}
and
\begin{align}\label{eq:repairdiag}
G_{\lambda_{j,n(j)},p,\gamma}(V_{j,n(j)};Q_0)\leq \frac{C_{N,p,\gamma}^{\mathrm{cell}}+\frac{1}{j}}{\alpha_j^p}
+\frac{C_{N,\gamma}\rho_j}{\alpha_j^p}+\frac1j.
\end{align}
Finally, for any $j\in\mathbb{N}$, let
$\Lambda_j:=\lambda_{j,n(j)}$ and $V_j:=V_{j,n(j)}$.
Letting $j\to\infty$ in  \eqref{eq:repairdiag}, we conclude that
\begin{align*}
\limsup_{j\to\infty}G_{\Lambda_j,p,\gamma}(V_j;Q_0)\leq C_{N,p,\gamma}^{\mathrm{cell}}.
\end{align*}
On the other hand, since $C^{\mathrm{cell}}_{N,p,\gamma}=m_{p,\gamma}(\ell)$, it follows that
\begin{align*}
C_{N,p,\gamma}^{\mathrm{cell}}\leq\liminf_{j\to\infty}G_{\Lambda_j,p,\gamma}(V_j;Q_0)
\end{align*}
and hence \eqref{eq:pah1} holds. The other assertions of the present
proposition holds  naturally by the choice of $n(j)$. This completes the proof
of Proposition \ref{prop:prepair}.
\end{proof}

\subsection{Recovery Sequences for Continuous Piecewise Affine Functions}\label{subsec:parec}

In this subsection, we construct recovery sequences for continuous
piecewise affine functions. We first recall the concept of piecewise-affine functions.

\begin{definition}\label{def:pa}
Let $N\in\mathbb N$, $a\in\mathbb R^N\setminus\{\mathbf{0}\}$, and $c\in\mathbb R$. A set
of the form $\{x\in\mathbb R^N:a\cdot x\leq c\}$ is called a \emph{closed affine half-space} in $\mathbb R^N$.
A \emph{closed convex polyhedron} $P$ in $\mathbb R^N$ is an intersection of
finitely many closed affine half-spaces with nonempty interior.
Let $U\subset\mathbb R^N$ be a bounded convex open set. A continuous function
$w:U\to\mathbb R$ is called \emph{continuous piecewise affine} if there are finitely
many closed convex polyhedrons $\{P_i\}_{i=1}^s$ with pairwise disjoint interiors, vectors
$\{\xi_i\}_{i=1}^s$, and  numbers $\{b_i\}_{i=1}^s$ such that $U\subset\bigcup_{i=1}^s P_i$
and, for any  $i\in\{1,\ldots,s\}$ and any $x\in U\cap P_i$,
\begin{align}\label{eq:paaff}
w(x)=\xi_i\cdot x+b_i.
\end{align}
\end{definition}
\begin{remark}
Let $U\subset \mathbb{R}^N $ be a bounded convex open set and
$w:U\to\mathbb R$ be continuous  piecewise affine. By the  assumption that $U$ is
convex and \eqref{eq:paaff}, we conclude  that, for any $x,y\in U,$
\begin{align}\label{eq:palip}
|w(x)-w(y)|\leq L|x-y|,
\end{align}
where $L:=\max_{1\leq i\leq s}|\xi_i|$.
\end{remark}

By the following conclusion, we  find that the recovery sequences for piecewise-affine functions exists.
\begin{proposition}\label{prop:pa}
Let $N\in\mathbb N$, $p\in [1,\infty)$, $\gamma\in(0,\infty)$, and $\Omega\subset\mathbb{R}^N$ be a bounded open set.
Assume that $U\supset\overline\Omega$ is a bounded convex open set and
$w:U\to\mathbb R$ is continuous piecewise affine. Let
$\{\lambda_j\}_{j\in\mathbb N}\subset(0,\infty)$ satisfy $\lambda_j\to\infty$.
Then there exists $\{w_j\}_{j\in\mathbb N}\subset L^p(\Omega)$ satisfying
$w_j\to w$ in $L^p(\Omega)$ and
\begin{align}\label{eq:parec}
\limsup_{j\to\infty}G_{\lambda_j,p,\gamma}(w_j;\Omega)
\leq C_{N,p,\gamma}^{\mathrm{cell}}\int_\Omega|\nabla w|^p\,dx.
\end{align}
\end{proposition}

\begin{proof}
Let $\{P_i\}_{i=1}^s$, $\{\xi_i\}_{i=1}^s$, and
$\{b_i\}_{i=1}^s$ be as in Definition \ref{def:pa}.
For any $i\in\{1,\ldots,s\}$, let
$D_i:=\Omega\cap \operatorname{int}(P_i).$
Since the interiors of the polyhedra are pairwise disjoint and
$U\subset\bigcup_{i=1}^sP_i$, it follows that
\begin{align}\label{eq:papart}
 \left|\Omega\setminus\left(\bigcup_{i=1}^sD_i\right)\right|=0.
\end{align}
We first show the present proposition under the additional assumption
\begin{align}\label{eq:papos}
 a_*:=\min_{1\leq i\leq s}|\xi_i|>0.
\end{align}
Let $\{\Lambda_j\}_{j\in\mathbb{N}}\subset (0,\infty),$
$\{V_j\}_{j\in\mathbb{N}}\subset L^\infty (Q_0)$, and $\{\rho_j\}_{j\in\mathbb{N}}\subset (0,\infty)$
be as in Proposition \ref{prop:prepair}.
By the fact that $\lim_{j\to\infty}\lambda_j=\infty$,  we find that  there exists  a sequence
$\{k_j\}_{j\in\mathbb{N}}\subset \mathbb{N}$ such that
\begin{align}\label{eq:paidx}
\lim_{j\to\infty}k_j=\infty\quad\text{and}\quad
\lim_{j\to\infty}\frac{\Lambda_{k_j}}{\lambda_j}=0.
\end{align}
Indeed, for any $j\in\mathbb{N}$, define
\begin{align*}
\mathcal K_j:=\left\{k\in\mathbb N: k\leq j\text{ and }
\Lambda_k\leq\sqrt{\lambda_j}\right\}\quad\text{and}\quad
k_j:=\max_{k\in\mathcal{K}_j}k.
\end{align*}
Moreover, if $\mathcal K_j=\emptyset$, define $k_j:=1$.
Since $\lim_{j\to\infty}\lambda_j=\infty$, it follows that, for all
sufficiently large $j$, $\mathcal K_j\neq\emptyset$ and, for any given
$K\in\mathbb{N}$, $K\in\mathcal K_j$ and $k_j\geq K$ for all sufficiently
large $j$.  This implies that \eqref{eq:paidx} holds.
For any $j\in\mathbb N$, let $\widehat\Lambda_j:=\Lambda_{k_j}$,
$\widehat V_j:=V_{k_j}$, and $\widehat\rho_j:=\rho_{k_j}$.
Since $\lim_{j\to\infty}\widehat{\rho}_j=0,$
after discarding finitely many indices,  we may assume that, for any $j\in\mathbb{N},$
$0<\widehat\rho_j<\frac{1}{2}$. Let  $i\in\{1,\ldots,s\}$ and $j\in\mathbb{N}.$ Define
\begin{align}\label{eq:pah}
h_{j,i}:=\left(\frac{|\xi_i|^p\widehat\Lambda_j}{\lambda_j}\right)^{\frac{1}{\gamma}}.
\end{align}
By \eqref{eq:paidx}, we find that
\begin{align}\label{eq:pahzero}
\lim_{j\to\infty}\max_{1\leq i\leq s}h_{j,i}=0.
\end{align}
Choose  matrix $R_i\in\mathcal O(N)$ such that $R_ie_N=\frac{\xi_i}{|\xi_i|},$
where $e_N:=(0,\ldots,1)$. For any $k\in\mathbb Z^N$, define
\begin{align}\label{eq:shui}
Q_{j,i,k}:=h_{j,i}R_i(k+Q_0)=h_{j,i}R_ik+h_{j,i}R_iQ_0.
\end{align}
Let $\mathscr Q_{j,i}$ denote the finite family of those cubes whose
closures are contained in $D_i$. Let $Q\in\mathscr Q_{j,i}$. From \eqref{eq:shui}, we infer that
$Q=x_Q+h_{j,i}R_iQ_0$, where $x_Q$ is the center of $Q$.
Thus, for any $x\in Q$, there exists a unique $z\in Q_0$ such that
$x=x_Q+h_{j,i}R_iz$. In what follows, for $Q\in\mathscr Q_{j,i}$ and $x\in Q$, the
corresponding point $z\in Q_0$ is always understood to be given by
$z=R_i^{-1}\frac{x-x_Q}{h_{j,i}}$, or equivalently $x=x_Q+h_{j,i}R_iz$.
Recall that, for any $z\in\mathbb{R}^N,$ $\ell(z):=z_N.$
By \eqref{eq:paaff} and the definition of $R_i$, we conclude that, for any $x\in Q$,
\begin{align}\label{eq:biaodashi1}
w(x)&=w(x_Q)+\xi_i\cdot(x-x_Q)=w(x_Q)+h_{j,i}\xi_i\cdot R_i z
=w(x_Q)+h_{j,i}(R_i^{-1}\xi_i)\cdot z\notag\\
&=w(x_Q)+|\xi_i|h_{j,i}e_N\cdot z=w(x_Q)+|\xi_i|h_{j,i}\ell(z).
\end{align}
Now define $w_j$ as follows.  For any $i\in\{1,\ldots,s\}$, $Q\in\mathscr Q_{j,i}$,
and $x\in Q$, let
\begin{align}\label{eq:padef}
w_j(x)&:=w(x_Q)+|\xi_i|h_{j,i}\widehat V_j
\left(R_i^{-1}\frac{x-x_Q}{h_{j,i}}\right)=w(x_Q)+|\xi_i|h_{j,i}\widehat V_j(z),
\end{align}
and, for any
\begin{align*}
x\in\Omega\setminus\left(\bigcup_{i=1}^s\bigcup_{Q\in\mathscr Q_{j,i}}Q\right),
\end{align*}
let $w_j(x)=w(x)$. Let $n_{j,i}:=\#\mathscr Q_{j,i},$ which denotes the number of cubes in the finite
family $\mathscr Q_{j,i}$. Since $\mathscr Q_{j,i}$ consists of cubes of side length $h_{j,i}$
whose closures are contained in $D_i$, it follows that, for any $i\in\{1,\ldots,s\}$,
\begin{align*}
D_i\setminus\left(\bigcup_{Q\in\mathscr Q_{j,i}}\overline{Q}\right)\subset\left\{
x\in D_i:\operatorname{dist}(x,\mathbb R^N\setminus D_i)\leq\sqrt N\,h_{j,i}\right\}.
\end{align*}
This, together with \eqref{eq:pahzero}, implies that
\begin{align}\label{eq:pafill}
\lim_{j\to\infty}\left|D_i\setminus\left(\bigcup_{Q\in\mathscr Q_{j,i}}Q\right)\right|=0
\end{align}
and hence
\begin{align}\label{eq:pavol}
\lim_{j\to\infty}n_{j,i}h_{j,i}^N=|D_i|.
\end{align}
From this,  the change-of-variables formula, \eqref{eq:padef}, \eqref{eq:pahzero},
and the fact that $\lim_{j\to\infty}\widehat{V}_j=\ell$ in $L^p(Q_0),$ we deduce that
\begin{align*}
\|w_j-w\|_{L^p(\Omega)}^p&=\sum_{i=1}^s\sum_{Q\in\mathscr Q_{j,i}}
|\xi_i|^ph_{j,i}^{N+p}\|\widehat V_j-\ell\|_{L^p(Q_0)}^p\\
&=\sum_{i=1}^s|\xi_i|^pn_{j,i}h_{j,i}^{N+p}\|\widehat V_j-\ell\|_{L^p(Q_0)}^p\longrightarrow0
\end{align*}
as $j\to\infty$. Thus,  to prove the  present proposition under the assumption \eqref{eq:papos},
it remains to show \eqref{eq:parec}. Let
\begin{align*}
Q^{\rm in}:=\left\{x\in Q:\operatorname{dist}(x,\partial Q)\geq
\widehat\rho_jh_{j,i}\right\}
\end{align*}
and
\begin{align*}
E_j:=\Omega\setminus\left(\bigcup_{i=1}^s\bigcup_{Q\in\mathscr Q_{j,i}}Q^{\rm in}\right).
\end{align*}
Observe that, for any $x\in Q$ with $Q\in\mathscr Q_{j,i}$,
\begin{align*}
\operatorname{dist}(x,\partial Q)=h_{j,i}\operatorname{dist}(z,\partial Q_0).
\end{align*}
This, combined with \eqref{eq:pah2}, implies that, for any $x\in Q$
with $\operatorname{dist}(x,\partial Q)<\widehat\rho_jh_{j,i},$
$w_j(x)=w(x)$ and hence
\begin{align}\label{eq:pae}
w_j=w\quad\text{on }E_j.
\end{align}
Note that,  for any $Q\in\mathscr Q_{j,i}$,
\begin{align*}
|Q\setminus Q^{\rm in}|=h_{j,i}^N\left[1-(1-2\widehat\rho_j)^N\right].
\end{align*}
Using this and \eqref{eq:papart}, we obtain that
\begin{align*}
|E_j|&=\sum_{i=1}^s\left|D_i\setminus\bigcup_{Q\in\mathscr Q_{j,i}}Q\right|
+\sum_{i=1}^s\sum_{Q\in\mathscr Q_{j,i}}|Q\setminus Q^{\rm in}|\\
&=\sum_{i=1}^s\left|D_i\setminus\bigcup_{Q\in\mathscr Q_{j,i}}Q\right|
+\left[1-(1-2\widehat\rho_j)^N\right]\sum_{i=1}^s n_{j,i}h_{j,i}^N.
\end{align*}
This, together with \eqref{eq:pafill}, \eqref{eq:pavol}, and the fact
that $\widehat\rho_j\to0$, implies that
\begin{align}\label{eq:haozi}
\lim_{j\to\infty}|E_j|=0.
\end{align}
Let
$$\mathcal A_j:=\bigcup_{i=1}^s\bigcup_{Q\in\mathscr Q_{j,i}}(Q\times Q)$$ and
\begin{align*}
\mathcal T_j:=\left\{(x,y)\in\Omega\times\Omega:\ x\neq y,\
\frac{|w_j(x)-w_j(y)|^p}{|x-y|^{p+\gamma}}\geq\lambda_j\right\}.
\end{align*}
We claim that, for all sufficiently large $j$,
\begin{align}\label{eq:painc}
\mathcal T_j\subset\mathcal A_j\cup(E_j\times E_j).
\end{align}
Assuming this claim for the moment, we continue with the proof. Let $j$ be  as above.
Using Lemma~\ref{lem:scaling} and \eqref{eq:pah}, we have
\begin{align}\label{eq:pacell}
\lambda_j\iint_{\mathcal{T}_j\cap (Q\times Q)}|x-y|^{\gamma-N}\,dx\,dy
=G_{\lambda_j,p,\gamma}(w_j;Q)
=|\xi_i|^ph_{j,i}^NG_{\widehat\Lambda_j,p,\gamma}(\widehat V_j;Q_0).
\end{align}
On the other hand, by \eqref{eq:pae} and \eqref{eq:palip}, we conclude that
\begin{align*}
&\lambda_j\iint_{\mathcal T_j\cap(E_j\times E_j)}|x-y|^{\gamma-N}\,dx\,dy\\
&\quad=\lambda_j\int_{E_j}\int_{E_j}\mathbf 1_{\left\{(x,y)\in E_j\times E_j:\ x\neq y,\ \frac{|w(x)-w(y)|^p}{|x-y|^{p+\gamma}}\geq \lambda_j\right\}}
|x-y|^{\gamma-N}\,dy\,dx\\
&\quad\leq\lambda_j\int_{E_j}\int_{E_j}\mathbf 1_{\left\{(x,y)\in E_j\times E_j:\ x\neq y,\ \frac{L^p}{|x-y|^{\gamma}}\geq \lambda_j\right\}}
|x-y|^{\gamma-N}\,dy\,dx
\leq\frac{\omega_N}{\gamma}L^p|E_j|.
\end{align*}
From this, \eqref{eq:painc}, and \eqref{eq:pacell},   we infer that
\begin{align*}
&G_{\lambda_j,p,\gamma}(w_j;\Omega)\\
&\quad
\leq
\lambda_j\iint_{\mathcal T_j\cap\mathcal A_j}|x-y|^{\gamma-N}\,dx\,dy
+\lambda_j\iint_{\mathcal T_j\cap(E_j\times E_j)}|x-y|^{\gamma-N}\,dx\,dy\\
&\quad\leq
\sum_{i=1}^s\sum_{Q\in\mathscr Q_{j,i}}\lambda_j
\iint_{\mathcal{T}_j\cap (Q\times Q)}|x-y|^{\gamma-N}\,dx\,dy\notag
+\frac{\omega_N}{\gamma}L^p|E_j|\\
&\quad\leq G_{\widehat\Lambda_j,p,\gamma}(\widehat V_j;Q_0)
\sum_{i=1}^s|\xi_i|^pn_{j,i}h_{j,i}^N
+\frac{\omega_N}{\gamma}L^p|E_j|.
\end{align*}
Letting $j\to\infty$ and using  \eqref{eq:pah1}, \eqref{eq:pavol}, and \eqref{eq:haozi}, we find that
\begin{align*}
&\limsup_{j\to\infty}G_{\lambda_j,p,\gamma}(w_j;\Omega)\\
&\quad\leq\lim_{j\to\infty}G_{\widehat\Lambda_j,p,\gamma}(\widehat V_j;Q_0)
\sum_{i=1}^s|\xi_i|^pn_{j,i}h_{j,i}^N
+\lim_{j\to\infty}\frac{\omega_N}{\gamma}L^p|E_j|\\
&\quad=C_{N,p,\gamma}^{\mathrm{cell}}\sum_{i=1}^s|\xi_i|^p|D_i|
=C_{N,p,\gamma}^{\mathrm{cell}}\int_\Omega|\nabla w|^p\,dx.
\end{align*}
This proves \eqref{eq:parec} under the assumption \eqref{eq:papos}.

Now, we show the claim \eqref{eq:painc}. Let $(x,y)\in \mathcal T_j $.
Suppose, to the contrary, that $ (x,y)\not\in\mathcal A_j\cup(E_j\times E_j).$
By symmetry, we may assume that $x\notin E_j$.  Then there exist $i\in\{1,\ldots,s\}$
and $Q\in\mathscr Q_{j,i}$ such that $x\in Q^{\rm in}$. Since $(x,y)\notin\mathcal A_j$,
it follows that  $y\notin Q$.  This implies that
\begin{align}\label{eq:ds}
|x-y|\geq\operatorname{dist}(x,\partial Q)\geq\widehat\rho_jh_{j,i}.
\end{align}
On the other hand, from \eqref{eq:padef} and \eqref{eq:biaodashi1}, it follows that
\begin{align*}
|w_j(x)-w(x)|\leq|\xi_i|h_{j,i}\|\widehat V_j-\ell\|_{L^\infty(Q_0)}.
\end{align*}
Using this, \eqref{eq:ds}, and the fact that $\lambda_jh_{j,i}^\gamma=|\xi_i|^p\widehat\Lambda_j$, we obtain that
\begin{align}\label{eq:wewill}
\frac{|w_j(x)-w(x)|^p}{\lambda_j|x-y|^{p+\gamma}}\leq\frac{
|\xi_i|^ph_{j,i}^p\|\widehat V_j-\ell\|_{L^\infty(Q_0)}^p}{
\lambda_j\widehat\rho_j^{p+\gamma}h_{j,i}^{p+\gamma}}
=\frac{|\xi_i|^p\|\widehat V_j-\ell\|_{L^\infty(Q_0)}^p}{
\lambda_j\widehat\rho_j^{p+\gamma}h_{j,i}^\gamma}
=\frac{\|\widehat V_j-\ell\|_{L^\infty(Q_0)}^p}{\widehat\Lambda_j\widehat\rho_j^{p+\gamma}}.
\end{align}
If $y\in E_j$, then, by \eqref{eq:pae}, we find that $w_j(y)-w(y)=0$.  If $y\notin E_j$, there exist $r\in\{1,\ldots,s\}$ and
$Q'\in\mathscr Q_{j,r}$ such that $y\in(Q')^{\rm in}$. From this and an argument similar to that used in  the
proof of \eqref{eq:wewill}, we deduce that
\begin{align*}
\frac{|w_j(y)-w(y)|^p}{\lambda_j|x-y|^{p+\gamma}}\leq
\frac{\|\widehat V_j-\ell\|_{L^\infty(Q_0)}^p}
 {\widehat\Lambda_j\widehat\rho_j^{p+\gamma}}.
\end{align*}
By this, \eqref{eq:wewill}, \eqref{eq:ds}, and  \eqref{eq:palip}, we conclude that
\begin{align*}
\frac{|w_j(x)-w_j(y)|^p}{\lambda_j|x-y|^{p+\gamma}}&\leq3^{p-1}
\left(\frac{L^p}{\lambda_j|x-y|^\gamma}+\frac{|w_j(x)-w(x)|^p}{\lambda_j|x-y|^{p+\gamma}}
+\frac{|w_j(y)-w(y)|^p}{\lambda_j|x-y|^{p+\gamma}}\right)\\
&\leq3^{p-1}\left(\frac{L^p}{a_*^p\widehat\Lambda_j\widehat\rho_j^\gamma}
+\frac{2\|\widehat V_j-\ell\|_{L^\infty(Q_0)}^p}{\widehat\Lambda_j\widehat\rho_j^{p+\gamma}}
\right)=:\varepsilon_j,
\end{align*}
where $L:=\max_{1\leq i\leq s}|\xi_i|$ is as in \eqref{eq:palip}.
Using this and \eqref{eq:pah3}, we have $\lim_{j\to\infty}\varepsilon_j=0.$ However, since
$(x,y)\in\mathcal T_j$, it follows that
\begin{align*}
1\leq\frac{|w_j(x)-w_j(y)|^p}{\lambda_j|x-y|^{p+\gamma}}.
\end{align*}
This contradicts  that $\lim_{j\to\infty}\varepsilon_j=0.$ Thus, $(x,y)\in \mathcal A_j\cup(E_j\times E_j)$.
This proves  the claim \eqref{eq:painc}.

Finally, we remove the assumption \eqref{eq:papos}. Fix $q\in\mathbb R^N\setminus\{{\bf0}\}$.
Choose a sequence $\{\varepsilon_m\}_{m\in\mathbb{N}}\subset(0,\infty)$ such that
$\lim_{m\to\infty}\varepsilon_m=0$ and, for any $m\in\mathbb{N}$ and $i\in\{1,\ldots,s\}$,
$\xi_i+\varepsilon_mq\neq0$. For any $m\in\mathbb{N}$ and $x\in U,$ define
\begin{align*}
w^{(m)}(x):=w(x)+\varepsilon_mq\cdot x.
\end{align*}
Then $w^{(m)}\to w$ in $W^{1,p}(\Omega)$. Moreover, for any given
$m\in\mathbb{N},$ $w^{(m)}$ is continuous finite piecewise affine  and
\begin{align*}
a_*^{(m)}:=\min_{1\leq i\leq s}|\xi_i+\varepsilon_mq|>0.
\end{align*}
Thus, for any given $m\in\mathbb{N},$ there exists a sequence
$\{w_j^{(m)}\}_{j\in\mathbb N}\subset L^p(\Omega)$ such that
$w_j^{(m)}\to w^{(m)}$ in $L^p(\Omega)$ and
\begin{align*}
\limsup_{j\to\infty}G_{\lambda_j,p,\gamma}(w_j^{(m)};\Omega)\leq
C_{N,p,\gamma}^{\mathrm{cell}}\int_\Omega|\nabla w^{(m)}|^p\,dx.
\end{align*}
Then, for any given $m\in\mathbb{N},$ there exists $J_m\in\mathbb{N}$
such that, for any $j\geq J_m$, $\|w_j^{(m)}-w^{(m)}\|_{L^p(\Omega)}\leq\frac1m$ and
\begin{align*}
G_{\lambda_j,p,\gamma}(w_j^{(m)};\Omega)\leq C_{N,p,\gamma}^{\mathrm{cell}}
\int_\Omega\left|\nabla w^{(m)}\right|^p\,dx+\frac1m.
\end{align*}
Without loss of generality, we may assume that $\{J_m\}_{m\in\mathbb{N}}$ is strictly increasing.
For any given $j\geq J_1$, let $m(j)$ be the unique integer such that $J_{m(j)}\leq j<J_{m(j)+1}.$
Since $\{J_m\}_{m\in\mathbb{N}}$ is strictly increasing, it follows that $\lim_{j\to\infty}m(j)=\infty.$
For any $j\geq J_1$, define $w_j:=w_j^{(m(j))}.$ Then, we have
\begin{align*}
\|w_j-w\|_{L^p(\Omega)}&\leq\left\|w_j-w^{(m(j))}\right\|_{L^p(\Omega)}+
\left\|w^{(m(j))}-w\right\|_{L^p(\Omega)}\\
&\leq\frac1{m(j)}+\left\|w^{(m(j))}-w\right\|_{L^p(\Omega)}
\end{align*}
and
\begin{align*}
G_{\lambda_j,p,\gamma}(w_j;\Omega)\leq C_{N,p,\gamma}^{\mathrm{cell}}\int_\Omega
|\nabla w^{(m(j))}|^p\,dx+\frac1{m(j)}.
\end{align*}
Letting $j\to\infty$, we find that $\lim_{j\to\infty}w_j=w$ in $L^p(\Omega)$ and \eqref{eq:parec} holds.
This completes the proof of Proposition \ref{prop:pa}.
\end{proof}

\subsection{Proofs of Theorems \ref{thm:plimsup} and \ref{thm:onelimsup}\label{proofsup}}

In this subsection, we give the proofs of Theorems \ref{thm:plimsup} and \ref{thm:onelimsup}.
The following standard density result follows from \cite[Proposition~2.8]{et99} (see also
\cite[Theorem~1]{vs14}).
\begin{lemma}\label{lem:padensity}
Let $N\in\mathbb N$, $p\in [1,\infty)$, $\Omega\subset\mathbb R^N$ be a bounded open interval
when $N=1$ or a bounded Lipschitz domain when $N\ge2$, and $u\in W^{1,p}(\Omega)$.  There exist a bounded
convex open set $U$ containing $\overline\Omega$ and a sequence
$\{u_m\}_{m\in\mathbb N}\subset C(U)$ of continuous finite
piecewise-affine functions such that $\lim_{m\to\infty}u_m= u$ in $W^{1,p}(\Omega).$
\end{lemma}

Now,  we show Theorem \ref{thm:plimsup} by using Propositions \ref{prop:pa} and \ref{prop:prepair}
and Lemma \ref{lem:padensity}.

\begin{proof}[Proof of Theorem \ref{thm:plimsup}]
Let $U$ and $\{u_m\}_{m\in\mathbb{N}}$ be as in Lemma \ref{lem:padensity}.
By Propositions \ref{prop:pa} and \ref{prop:prepair}, we  conclude that, for any given
$m\in\mathbb{N}$, there exists a sequence $\{u_j^{(m)}\}_{j\in\mathbb{N}}\subset L^p(\Omega)$ such that
$u_j^{(m)}\to u_m$ in $L^p(\Omega)$ and
\begin{align*}
\limsup_{j\to\infty} G_{\lambda_j,p,\gamma}(u_j^{(m)};\Omega)\leq
C_{N,p,\gamma}^{\mathrm{cell}}\int_\Omega|\nabla u_m|^p\,dx.
\end{align*}
Thus, for any given $m\in\mathbb{N},$ there  exists  $J_m\in\mathbb{N}$ such that, for any $j\geq J_m$,
$\|u_j^{(m)}-u_m\|_{L^p(\Omega)}\leq\frac1m$ and
\begin{align*}
G_{\lambda_j,p,\gamma}(u_j^{(m)};\Omega)\leq C_{N,p,\gamma}^{\mathrm{cell}}
\int_\Omega\left|\nabla u_m\right|^p\,dx+\frac1m.
\end{align*}
Without loss of generality, we may assume that $\{J_m\}_{m\in\mathbb{N}}$ is strictly increasing.
For any given  $j\geq J_1$, let $m(j)$ be the unique integer such that $J_{m(j)}\leq j<J_{m(j)+1}.$
Since $\{J_m\}_{m\in\mathbb{N}}$ is strictly increasing, it follows that  $\lim_{j\to\infty}m(j)=\infty.$
For any $j\geq J_1$, define $u_j:=u_j^{(m(j))}.$ Then, we have
\begin{align*}
\|u_j-u\|_{L^p(\Omega)}\leq\left\|u_j-u_{m(j)}\right\|_{L^p(\Omega)}
+\left\|u_{m(j)}-u\right\|_{L^p(\Omega)}\leq\frac1{m(j)}+\left\|u_{m(j)}-u\right\|_{L^p(\Omega)}
\end{align*}
and
\begin{align*}
G_{\lambda_j,p,\gamma}(u_j;\Omega)\leq C_{N,p,\gamma}^{\mathrm{cell}}
\int_\Omega\left|\nabla u_{m(j)}\right|^p\,dx+\frac1{m(j)}.
\end{align*}
Letting $j\to\infty$, we find that $\lim_{j\to\infty}u_j=u$ in $L^p(\Omega)$ and \eqref{eq:jiahua} holds.
This completes the proof of Theorem \ref{thm:plimsup}.
\end{proof}

The following standard strict approximation result is just \cite[Theorem~5.3]{eg15}.
\begin{lemma}\label{lem:strictapp}
Let $N\in\mathbb N$, $\Omega\subset \mathbb{R}^N$ be a open set and $u\in BV(\Omega)$.
Then there  exists a sequence $\{u_m\}_{m\in\mathbb{N}}\subset W^{1,1}(\Omega)\cap C^\infty(\Omega)$ such that
$\lim_{m\to\infty}u_m=u$ in $L^1(\Omega)$ and $\lim_{m\to\infty}\int_{\Omega}|\nabla u_m(x)|\,dx=|Du|(\Omega)$.
\end{lemma}

\begin{proof}[Proof of Theorem \ref{thm:onelimsup}]
Applying Theorem \ref{thm:plimsup} and Lemma \ref{lem:strictapp}  and using
the same diagonal argument used in the proof of Theorem \ref{thm:plimsup},
we conclude that the desired conclusion  holds; we omit the details here. This
completes the proof of Theorem \ref{thm:onelimsup}.
\end{proof}

\section{The Gamma-Liminf Inequality\label{sec:liminf}}

In this section, we prove the Gamma-liminf inequality stated  in Theorem \ref{thm:main}.
We have the  following  conclusions.

\begin{theorem}\label{thm:pliminf}
Let $N\in\mathbb N$, $p\in(1,\infty)$, $\gamma\in(0,\infty)$, and $\Omega\subset\mathbb R^N$ be a bounded open set.
Let $\{\lambda_j\}_{j\in\mathbb N}\subset(0,\infty)$ satisfy
$\lambda_j\to\infty$ and $\{u_j\}_{j\in\mathbb N}\subset L^p(\Omega)$
satisfy $u_j\to u$ in $L^p(\Omega)$ for some $u\in L^p(\Omega)$. Then
\begin{align}\label{eq:pliminf}
\liminf_{j\to\infty}G_{\lambda_j,p,\gamma}(u_j;\Omega)\geq C_{N,p,\gamma}^{\mathrm{cell}}
\int_\Omega|\nabla u|^p\,dx,
\end{align}
where the right-hand side is understood to be $\infty$ if $u\notin W^{1,p}(\Omega)$.
\end{theorem}

\begin{theorem}\label{thm:oneliminf}
Let $N\in\mathbb N$, $\gamma\in(0,\infty)$, and $\Omega\subset \mathbb{R}^N$ be a bounded open set.
Let $\{\lambda_j\}_{j\in\mathbb N}\subset(0,\infty)$ satisfy
$\lambda_j\to\infty$ and $\{u_j\}_{j\in\mathbb N}\subset L^1(\Omega)$ satisfy
$u_j\to u$ in $L^1(\Omega)$ for some $u\in L^1(\Omega)$.  Then
\begin{align}\label{eq:oneliminf}
\liminf_{j\to\infty}G_{\lambda_j,1,\gamma}(u_j;\Omega)\geq C_{N,1,\gamma}^{\mathrm{cell}}|Du|(\Omega),
\end{align}
where the right-hand side is $\infty$ when $u\notin BV(\Omega)$.
\end{theorem}
The proofs of Theorems \ref{thm:pliminf} and \ref{thm:oneliminf} are respectively given
in Subsections \ref{proofp} and \ref{proofone}.

\subsection{Proof of Theorem \ref{thm:pliminf}\label{proofp}}
In this subsection, we show Theorem \ref{thm:pliminf}.
Throughout this subsection, we use the following notation. For any given $R\in\mathcal O(N)$,
$x\in\mathbb R^N$, and $r\in (0,\infty)$, define
\begin{align*}
 Q_r^R(x):=x+rRQ_0.
\end{align*}
The following Morse type covering lemma is  a special case of
 \cite[Theorem 5.51]{afp00}.
\begin{lemma}\label{lem:morse}
Let $\mu$ be a finite positive Radon measure on a bounded open set
$U\subset\mathbb R^N$ and $A_0\subset U$ be a $\mu$-measurable set.  Let
$V$ be a family of open cubes of the form $Q_r^R(x)$, where $x\in A_0$, $r\in(0,\infty)$,
and $R\in\mathcal O(N)$. Assume that, for any cube $Q\in V$,
$\overline{Q}\subset U$ and $\mu(\partial Q)=0.$ Assume further that, for any given  $x\in A_0$
and any given  $\delta\in(0,\infty)$, the family $V$ contains a cube centered at $x$ and having
diameter smaller than $\delta$.  Then there exists a countable subfamily
$(Q_i)_{i\in\mathcal I}\subset V$ of pairwise disjoint cubes such that
\begin{align*}
 \mu\left(A_0\setminus\left(\bigcup_{i\in\mathcal I}Q_i\right)\right)=0.
\end{align*}
\end{lemma}

We prove Theorem \ref{thm:pliminf} by using Lemma \ref{lem:morse}.

\begin{proof}[Proof of Theorem \ref{thm:pliminf}]
If $u\notin W^{1,p}(\Omega)$, then \eqref{eq:pliminf} follows directly from Lemma \ref{lem:gp}.
Assume  that $u\in W^{1,p}(\Omega)$ and define the measure $d\mu(x):=|\nabla u(x)|^p\,dx$
on $\Omega.$ Without loss of generality, we may assume that $\mu(\Omega)>0.$
Let
\begin{align*}
A_0:=\left\{x\in\Omega:\nabla u(x)\neq0,\ \lim_{\rho\to0}
\fint_{B(x,\rho)}|\nabla u(y)-\nabla u(x)|^p\,dy=0\right\}.
\end{align*}
By the Lebesgue differentiation theorem, we find that  $\mu(\Omega\setminus A_0)=0$.
Let $x\in A_0$ and $\xi:=\nabla u(x).$ Choose $R_x\in\mathcal O(N)$ such that $ R_xe_N=\frac{\xi}{|\xi|}$ and
choose $ r_x\in(0,\infty)$ such that $Q_r^{R_x}(x)\subset\Omega$ for any $r\in(0, r_x)$.
Let $r\in (0,r_x)$. For any $j\in\mathbb{N}$ and $z\in Q_0$,  define $b_{x,r}:=\fint_{Q_r^{R_x}(x)}u(y)\,dy$,
\begin{align*}
u_{x,r}(z):=\frac{u(x+rR_xz)-b_{x,r}}{|\xi|r},\ \ \text{and}\ \ u_{j,x,r}(z):
=\frac{u_j(x+rR_xz)-b_{x,r}}{|\xi|r}.
\end{align*}
Since  $u_j\to u$ in $L^p(\Omega)$, it follows that $u_{j,x,r}\to u_{x,r}$ in $L^p(Q_0).$
Using this, the definition of $m_{p,\gamma}$, and  Lemma \ref{lem:scaling}, we conclude that
\begin{align}\label{eq:yongyou}
\liminf_{j\to\infty}G_{\lambda_j,p,\gamma}\left(u_j;Q_r^{R_x}(x)\right)&=|\xi|^pr^N
\liminf_{j\to\infty}G_{\frac{\lambda_jr^\gamma}{|\xi|^p},p,\gamma}\left(u_{j,x,r};Q_0\right)\notag\\
&\geq|\xi|^pr^Nm_{p,\gamma}(u_{x,r}).
\end{align}
Fix $\varepsilon\in(0,\min\{1,C_{N,p,\gamma}^{\mathrm{cell}}/2\})$.
By the Lebesgue differentiation theorem again, we conclude that
\begin{align*}
\lim_{r\to 0}\frac{\mu(Q_r^{R_x}(x))}{r^N}=|\nabla u(x)|^p=|\xi|^p.
\end{align*}
Thus, after decreasing $r_x$ if necessary, we have
\begin{align*}
\mu(Q_r^{R_x}(x))\leq(1+\varepsilon)|\xi|^p r^N.
\end{align*}
For any $y\in\Omega$, let $f(y):=u(y)-\xi\cdot(y-x).$ From the definition of $b_{x,r}$, we deduce that
\begin{align*}
\fint_{Q_r^{R_x}(x)}f(y)\,dy=\fint_{Q_r^{R_x}(x)}u(y)\,dy
-\xi\cdot\fint_{Q_r^{R_x}(x)}(y-x)\,dy=b_{x,r}.
\end{align*}
Using this and applying the  Poincar\'e inequality to $f$ on $Q_r^{R_x}(x)$, we find that
\begin{align}\label{eq:lushang}
\int_{Q_r^{R_x}(x)}\left|f(y)-b_{x,r}\right|^p\,dy\lesssim r^p\int_{Q_r^{R_x}(x)}
|\nabla u(y)-\xi|^p\,dy.
\end{align}
Since $R_xe_N=\frac{\xi}{|\xi|},$ it follows that $R_x^T\xi=|\xi|e_N.$
Consequently, for any $z\in Q_0$,
\begin{align*}
\xi\cdot(rR_xz)=r(R_x^T\xi)\cdot z=|\xi|r\,e_N\cdot z=|\xi|r z_N=|\xi|r\ell(z)
\end{align*}
and hence
\begin{align*}
u_{x,r}(z)-\ell(z)=\frac{u(x+rR_xz)-b_{x,r}}{|\xi|r}-\ell(z)
=\frac{u(x+rR_xz)-b_{x,r}-\xi\cdot(rR_xz)}{|\xi|r}.
\end{align*}
From this, the change of variables $y=x+rR_xz,$ and \eqref{eq:lushang}, we infer that
\begin{align*}
\|u_{x,r}-\ell\|_{L^p(Q_0)}^p&=\frac{1}{|\xi|^pr^p}\int_{Q_0}\left|u(x+rR_xz)-b_{x,r}
-\xi\cdot(rR_xz)\right|^p\,dz\\
&=\frac{1}{|\xi|^pr^{N+p}}\int_{Q_r^{R_x}(x)}\left|u(y)-b_{x,r}
-\xi\cdot(y-x)\right|^p\,dy\\
&\lesssim\frac{1}{|\xi|^pr^N}\int_{Q_r^{R_x}(x)}|\nabla u(y)-\xi|^p\,dy.
\end{align*}
Since $x$ is a $L^p$-Lebesgue  point of $\nabla u$, it follows that
\begin{align*}
\lim_{r\to 0}u_{x,r}=\ell\ \ \text{in}\ \ L^p(Q_0).
\end{align*}
Using  this and Lemma \ref{lem:mplsc}, we conclude that
\begin{align*}
C_{N,p,\gamma}^{\mathrm{cell}}=m_{p,\gamma}(\ell)\leq\liminf_{r\to 0}m_{p,\gamma}(u_{x,r}).
\end{align*}
By this, after decreasing $r_x$ if necessary, we find that
\begin{align}\label{eq:mplocal}
C_{N,p,\gamma}^{\mathrm{cell}}-\varepsilon\leq m_{p,\gamma}(u_{x,r}).
\end{align}
From this,\eqref{eq:mplocal}, and \eqref{eq:yongyou}, we deduce that
\begin{align}
\liminf_{j\to\infty}G_{\lambda_j,p,\gamma}\left(u_j;Q_r^{R_x}(x)\right)\geq|\xi|^pr^Nm_{p,\gamma}(u_{x,r})
\geq\frac{C_{N,p,\gamma}^{\mathrm{cell}}-\varepsilon}{1+\varepsilon}
\mu\left(Q_r^{R_x}(x)\right).\label{eq:plocal}
\end{align}

Now, define
\begin{align*}
V_\varepsilon:=\left\{Q_r^{R_x}(x):x\in A_0,\ r\in(0,r_x)\right\}.
\end{align*}
Applying Lemma \ref{lem:morse} to $(\mu,\Omega, V_\varepsilon),$ we conclude that there  exists
a  countable subfamily $\{Q_i\}_{i\in\mathcal{I}}\subset V_\varepsilon$ of pairwise disjoint cubes such that
\begin{align*}
\mu\left(A_0\setminus\left(\bigcup_{i\in\mathcal I}Q_i\right)\right)=0.
\end{align*}
This, combined with the fact that $\mu(\Omega\setminus A_0)=0$, implies
\begin{align*}
\mu\left(\Omega\setminus\left(\bigcup_{i\in\mathcal I}Q_i\right)\right)=0.
\end{align*}
Since $\mu(\Omega)<\infty$, we can choose a finite subfamily of $\{Q_i\}_{i\in\mathcal{I}}$,
relabeled as $Q_1,\ldots,Q_M$, such that
\begin{align*}
\mu\left(\Omega\setminus\left(\bigcup_{i=1}^M Q_i\right)\right)<\varepsilon.
\end{align*}
Using this and  \eqref{eq:plocal}, we obtain that
\begin{align*}
\liminf_{j\to\infty}G_{\lambda_j,p,\gamma}\left(u_j;\Omega\right)
&\geq\liminf_{j\to\infty}\sum_{i=1}^M G_{\lambda_j,p,\gamma}\left(u_j;Q_i\right)\\
&\geq\frac{C_{N,p,\gamma}^{\mathrm{cell}}-\varepsilon}{1+\varepsilon}
\sum_{i=1}^M\mu(Q_i)>
\frac{C_{N,p,\gamma}^{\mathrm{cell}}-\varepsilon}{1+\varepsilon}(\mu(\Omega)-\varepsilon).
\end{align*}
Letting $\varepsilon\to0$, we find that \eqref{eq:pliminf} holds.
This completes the proof of Theorem \ref{thm:pliminf}.
\end{proof}

\subsection{One-Dimensional Limits of Rescaled BV Functions\label{subsec:bvlim}}

In this subsection, we study one-dimensional limits of rescaled BV functions.
Let $N\in\mathbb N$, $\Omega\subset\mathbb R^N$ be an open set,
$\{u_k\}_{k\in\mathbb N}\subset BV(\Omega)$, and $u\in BV(\Omega)$.
Recall that $u_k$ \emph{converges strictly }to $u$ in
$BV(\Omega)$ if $\lim_{k\to\infty}u_k=u$ in $L^1(\Omega)$ and  $\lim_{k\to\infty}|Du_k|(\Omega)=|Du|(\Omega)$,
and for any given $R\in\mathcal O(N)$, $x\in\mathbb R^N$, and $r\in (0,\infty)$,
$Q_r^R(x):=x+rRQ_0.$ Let $x\in\Omega$, $R\in\mathcal O(N)$, and $r\in(0,\infty)$ such that
$Q_r^R(x)\subset\Omega$ and $|Du|(Q_r^R(x))\in (0,\infty)$. For any $z\in Q_0,$ let
\begin{align}\label{eq:rescale}
u_{x,r}^R(z):=\frac{r^{N-1}}{|Du|(Q_r^R(x))}\left[u(x+rRz)-\fint_{Q_r^R(x)}u\,dy\right].
\end{align}
By a change of variables, we find that
\begin{align*}
\left|Du_{x,r}^R\right|(Q_0)=\frac{r^{N-1}}{|Du|(Q_r^R(x))}\frac{|Du|(Q_r^R(x))}{r^{N-1}}=1.
\end{align*}
Recall that $I_0:=(-\frac{1}{2},\frac{1}{2})$ and, for any $q\in BV(I_0)$ and $z\in Q_0$, $W_{q}(z):=q(z_N)$.
The following lemma is the main  result in this subsection.
\begin{lemma}\label{lem:tangent}
Let $N\in\mathbb N$, $\Omega\subset\mathbb R^N$ be a bounded open set, and
$u\in BV(\Omega)$. Then, for $|Du|$-a.e. $x\in\Omega$, there exist
an orthogonal matrix $R_x\in\mathcal O(N)$, a sequence
$\{r_k\}_{k\in\mathbb N}\subset(0,\infty)$ with $r_k\to 0$, and a
nondecreasing function $q_x\in BV(I_0)$ such that the following
statements hold.
\begin{enumerate}[{\rm(i)}]
\item For any $k\in\mathbb{N},$ $\overline{Q_{r_k}^{R_x}(x)}\subset\Omega$,
$|Du|(Q^{R_x}_{r_k}(x))\in (0,\infty)$, and
\begin{align*}
|Du|\left(\partial Q_{r_k}^{R_x}(x)\right)=0.
\end{align*}

\item  $|DW_{q_x}|(Q_0)=|Dq_x|(I_0)=1$.

\item
$\lim_{k\to\infty}u_{x,r_k}^{R_x}=W_{q_x}$ in $L^1(Q_0)$. Consequently, $u_{x,r_k}^{R_x}$
converges strictly to $W_{q_x}$ in $BV(Q_0)$ as $k\to\infty$.
\end{enumerate}
\end{lemma}

We first recall some notions used below. Let $\Omega\subset\mathbb R^N$ be open,
$u\in L^1_{\mathrm{loc}}(\Omega)$, and $x\in\Omega$.  Following \cite[Definition~3.63]{afp00}, a number
$a\in\mathbb R$ is the \emph{approximate limit} of $u$ at $x$ if
\begin{align*}
\lim_{r\to 0^+}\fint_{B(x,r)}|u(y)-a|\,dy=0.
\end{align*}
Such a number is unique whenever it exists and is denoted by
$\widetilde u(x)$.  Following \cite[Definition~3.70]{afp00}, at any
point $x\in\Omega$ at which $\widetilde u(x)$ exists, the function $u$ is \emph{approximately
differentiable} at $x$ if there exists $\xi\in\mathbb R^N$ such that
\begin{align*}
\lim_{r\to 0^+}\frac{1}{r}\fint_{B(x,r)}\left|u(y)-\widetilde u(x)-\xi\cdot(y-x)\right|\,dy
=0.
\end{align*}
The vector $\xi$ is unique and is denoted by $\nabla u(x)$, which is
called the \emph{approximate gradient} of $u$ at $x$.

For  any $x\in\Omega$, $r\in(0,\infty)$, and $\nu\in\mathbb S^{N-1}$, define
\begin{align*}
B^\pm(x,r,\nu):=\left\{y\in B(x,r):\ \pm(y-x)\cdot\nu>0\right\}.
\end{align*}
Following \cite[Definition~3.67]{afp00}, the \emph{approximate jump set}
$J_u$ consists of all points
$x\in\Omega$ for which there exist $u^+(x),u^-(x)\in\mathbb R$ and
$\nu_u(x)\in\mathbb S^{N-1}$ such that $u^+(x)\neq u^-(x)$,
\begin{align*}
\lim_{r\to0^+}\fint_{B^+(x,r,\nu_u(x))}|u(y)-u^+(x)|\,dy=0,
\end{align*}
and
\begin{align*}
\lim_{r\to0^+}\fint_{B^-(x,r,\nu_u(x))}|u(y)-u^-(x)|\,dy=0.
\end{align*}
The triplet $\left(u^+(x),u^-(x),\nu_u(x)\right)$ is unique up to interchanging $u^+(x)$
and $u^-(x)$ and reversing the direction of $\nu_u(x)$.

Recall that an $\mathcal H^{N-1}$-measurable set $E\subset\mathbb R^N$ is countably \emph{$(N-1)$-rectifiable}
if there exists a sequence of Lipschitz functions
$\{f_k\}_{k\in\mathbb{N}}$ mapping $\mathbb{R}^{N-1}$ to $\mathbb{R}^N$
such that
\begin{align*}
\mathcal H^{N-1}\left(E\setminus\bigcup_{k\in\mathbb{N}}f_k\left(\mathbb R^{N-1}\right)
\right)=0.
\end{align*}
In other words, except for an $\mathcal H^{N-1}$-null subset,
$E$ is contained in the union of countably many Lipschitz images of $\mathbb R^{N-1}$.
For any $\nu\in\mathbb S^{N-1}$, let
\begin{align*}
\nu^\perp:=\left\{z\in\mathbb R^N:z\cdot\nu=0\right\}.
\end{align*}
The following result is the classical structure theorem for BV functions
(see, for instance, \cite[Proposition~3.69, Theorems 3.77, 3.78, and~3.83, Definition~3.91, and
Proposition~3.92]{afp00}).
\begin{lemma}\label{lem:bvstruct}
Let $N\in\mathbb N$, $\Omega\subset\mathbb R^N$ be an open set, and $u\in BV(\Omega)$.
\begin{enumerate}[{\rm(i)}]
\item For $\mathcal L^N$-a.e. $x\in\Omega$, the approximate gradient $\nabla u(x)$ exists  and
$\nabla u\in L^1(\Omega;\mathbb R^N)$.

\item The jump set $J_u$ is countably $(N-1)$-rectifiable.  Moreover, for
$\mathcal H^{N-1}$-a.e. $x\in J_u$, the approximate tangent
hyperplane $T_xJ_u$ exists and $T_xJ_u=\nu_u(x)^\perp.$

\item There exist unique finite $\mathbb R^N$-valued Radon measures $D^ju$ and
$D^cu$, which are singular with respect to Lebesgue measure, such that
\begin{align*}
Du=\nabla u\,\mathcal L^N+D^ju+D^cu.
\end{align*}
Moreover, the jump part has the representation
\begin{align*}
D^ju=\left(u^+-u^-\right)\nu_u\left(\mathcal H^{N-1}\llcorner J_u\right).
\end{align*}

\item The positive measures $|\nabla u|\mathcal L^N$, $|D^ju|$, and
$|D^cu|$ are mutually singular, and hence
\begin{align}\label{eq:bvvar}
|Du|=|\nabla u|\,\mathcal L^N+|D^ju|+|D^cu|.
\end{align}
\end{enumerate}
\end{lemma}

A positive Radon measure $\mu$ on $\mathbb R^N$ is called
\emph{$(N-1)$-rectifiable} if there exist a countably
$(N-1)$-rectifiable set $E\subset\mathbb R^N$ and a Borel function
$\theta:E\to(0,\infty)$ such that
\begin{align*}
 \mu=\theta\left(\mathcal H^{N-1}\llcorner E\right).
\end{align*}
The function $\theta$ is called the \emph{density} of the rectifiable measure.
If $P\subset\mathbb R^N$ is an $(N-1)$-dimensional linear
subspace and $\theta\in(0,\infty)$, we say that $\mu$ has
\emph{approximate tangent hyperplane $P$ with density $\theta$ at
$x$} if, for any $\varphi\in C_{\rm c}(\mathbb R^N)$,
\begin{align*}
\lim_{r\to0^+}\frac{1}{r^{N-1}}\int_{\mathbb R^N}\varphi\left(\frac{y-x}{r}\right)\,d\mu(y)
=\theta\int_P\varphi(z)\,d\mathcal H^{N-1}(z).
\end{align*}
The following standard result describes the tangent hyperplane and density of
a rectifiable measure (see, for instance,  \cite[Theorem~2.83]{afp00}).
\begin{lemma}\label{lem:recttan}
Let $\mu:=\theta(\mathcal H^{N-1}\llcorner E)$ be an
$(N-1)$-rectifiable positive Radon measure on $\mathbb R^N$.
Then, for $\mu$-a.e. $x\in E$, the approximate tangent hyperplane $T_xE$
exists and $\mu$ has approximate tangent hyperplane $T_xE$ with density
$\theta(x)$ at $x$.  In particular,
\begin{align*}
\lim_{r\to0^+}\frac{\mu(B(x,r))}{r^{N-1}}=\theta(x)\omega_{N-1}.
\end{align*}
Here $\omega_{N-1}$ is the volume of the unit ball in
$\mathbb R^{N-1}$, with the convention $\omega_0=1$.
\end{lemma}

The following lemma is a special case of  the Besicovitch
differentiation theorem (see, for  instance, \cite[Theorem~2.22]{afp00}).
\begin{lemma}\label{lem:singdiff}
Let $\mu$ and $\lambda$ be finite positive Radon measures on an open set
$U\subset\mathbb R^N$.  If $\lambda\perp\mu$, then, for $\mu$-a.e. $x\in U$,
\begin{align*}
\lim_{r\to0^+}\frac{\lambda(B(x,r))}{\mu(B(x,r))}=0.
\end{align*}
\end{lemma}

For any $t\in I_0, $ define $q^{\rm ac}(t):=t$. Moreover, for any $t\in I_0$, let
\begin{align*}
q^{\rm j}(t)&:=\begin{cases}
-\frac{1}{2},&t\in(-\frac{1}{2},0),\\
0,&t=0,\\
\frac{1}{2},&t\in(0,\frac{1}{2}).
\end{cases}
\end{align*}
Since $q^{\mathrm{ac}}$ and $q^{\mathrm{j}}$ are nondecreasing, it follows that
\begin{align}\label{eq:profmass}
|Dq^{\rm ac}|(I_0)=|Dq^{\rm j}|(I_0)=1.
\end{align}
Recall that, for any given $q\in BV(I_0)$ and any $z\in Q_0$, $W_q(z):=q(z_N)$.
By a direct computation, we obtain that
\begin{align*}
DW_q=e_N\left[\left(\mathcal L^{N-1}\llcorner I_0^{N-1}\right)\otimes Dq\right]
\end{align*}
and hence
\begin{align*}
|DW_q|(Q_0)=|Dq|(I_0)
\end{align*}
This, together with \eqref{eq:profmass}, implies that
\begin{align*}
|DW_{q^{\mathrm{ac}}}|(Q_0)=|DW_{q^{\mathrm{j}}}|(Q_0)=1.
\end{align*}
We first consider the absolutely  continuous part. We have the  following conclusion.
\begin{lemma}\label{lem:acpar}
Let $N\in\mathbb N$, $\Omega\subset\mathbb{R}^N$ be an open set, and
$u\in BV(\Omega)$.  For $|\nabla u|\mathcal L^N$-a.e. $x\in\Omega$, there exists
$R_x\in\mathcal O(N)$ such that, for any sufficiently small $r\in(0,\infty),$
$Q_r^{R_x}(x)\subset\Omega$, $|Du|(Q_r^{R_x}(x))\in (0,\infty)$, and
\begin{align*}
\lim_{r\to0}u_{x,r}^{R_x}=W_{q^{\rm ac}}\quad\text{in }L^1(Q_0),
\end{align*}
where $u^{R_x}_{x,r}$ is defined as in \eqref{eq:rescale} with $R$ replaced by $R_x.$
\end{lemma}

\begin{proof}
Let $D^su:=D^ju+D^cu$. By Lemma~\ref{lem:bvstruct}(iii),
we find that $|D^su|\perp\mathcal L^N$.
Using this  and Lemma \ref{lem:singdiff}, we conclude  that, for $\mathcal L^N$-a.e.
$x\in\Omega$,
\begin{align}\label{eq:acsing}
\lim_{\rho\to0^+}\frac{|D^su|(B(x,\rho))}{\rho^N}=0.
\end{align}
From the Lebesgue differentiation theorem and the fact that $\nabla u\in L^1(\Omega;\mathbb R^N)$,
we infer that, for $\mathcal L^N$-a.e. $x\in\Omega,$ $x$ is a Lebesgue point of $\nabla u$. Moreover,
\begin{align*}
\left(|\nabla u|\mathcal L^N\right)\left(\left\{x\in\Omega:\nabla u(x)=0\right\}\right)
=\int_{\left\{y\in\Omega:\ \nabla u(y)=0\right\}}|\nabla u(y)|\,dy=0.
\end{align*}
Thus, for $|\nabla u|\mathcal L^N$-a.e. $x\in\Omega$, the vector
$\nabla u(x)$ is nonzero, $x$ is a Lebesgue point of $\nabla u$, and \eqref{eq:acsing} holds.
Fix such a point $x$ and choose $R_x\in\mathcal O(N)$ such that
$R_xe_N=\frac{\nabla u(x)}{|\nabla u(x)|}.$ Then
\begin{align}\label{eq:acfine}
\lim_{r\to0}\fint_{Q_r^{R_x}(x)}|\nabla u(y)-\nabla u(x)|\,dy=0\ \ \text{and}\ \
\lim_{r\to0}\frac{|D^su|(Q_r^{R_x}(x))}{r^N}=0.
\end{align}
Let $$A_r:=\frac{|Du|(Q_r^{R_x}(x))}{r^{N-1}}.$$
By Lemma~\ref{lem:bvstruct}(iv) and \eqref{eq:acfine}, we conclude that
\begin{align}\label{eq:acmass}
\lim_{r\to 0}\frac{A_r}{r}=\lim_{r\to 0}\fint_{Q_r^{R_x}(x)}|\nabla u(y)|\,dy
+\frac{|D^su|(Q_r^{R_x}(x))}{r^N}=|\nabla u(x)|.
\end{align}
Thus, for any sufficiently small $r\in(0,\infty)$, $A_r\in(0,\infty).$
For any $y\in\Omega$, define
$$g(y):=u(y)-\nabla u(x)\cdot(y-x).$$  Then
\begin{align*}
\fint_{Q_r^{R_x}(x)}g(y)\,dy=\fint_{Q_r^{R_x}(x)}u(y)\,dy=:b_r.
\end{align*}
Using this and the  Poincar\'e inequality, we find that, for any sufficiently small $r\in(0,\infty)$,
\begin{align*}
\frac{1}{A_r}\fint_{Q_r^{R_x}(x)}|g(y)-b_r|\,dy&\lesssim
\frac{1}{A_rr^{N-1}}|Dg|(Q_r^{R_x}(x))\\
&\sim\frac{r}{A_r}\left[\fint_{Q_r^{R_x}(x)}|\nabla u(y)-\nabla u(x)|\,dy
+\frac{|D^su|(Q_r^{R_x}(x))}{r^N}\right].
\end{align*}
This, together with \eqref{eq:acmass} and \eqref{eq:acfine}, implies that
\begin{align}\label{eq:acerror}
\lim_{r\to 0} \frac{1}{A_r}\fint_{Q_r^{R_x}(x)}|g(y)-b_r|\,dy=0.
\end{align}
From \eqref{eq:rescale} and the definitions of $u^{R_x}_{x,r}$, $g$, and $R_x$,
it follows that, for any sufficiently small $r\in(0,\infty)$ and  any $z\in Q_0,$
\begin{align*}
u_{x,r}^{R_x}(z)&=\frac{r^{N-1}}{|Du|(Q_r^{R_x}(x))}
\left[u(x+rR_xz)-b_r\right]=\frac{u(x+rR_xz)-b_r}{A_r}\\
&=\frac{g(x+rR_xz)-b_r}{A_r}+\frac{\nabla u(x)\cdot(rR_xz)}{A_r}\\
&=\frac{g(x+rR_xz)-b_r}{A_r}+\frac{|\nabla u(x)|rz_N}{A_r},
\end{align*}
which further implies that
\begin{align*}
\left\|u_{x,r}^{R_x}-W_{q^{\rm ac}}\right\|_{L^1(Q_0)}&\lesssim
\frac{1}{A_r}\fint_{Q_r^{R_x}(x)}|g-b_r|\,dy+
\left|\frac{|\nabla u(x)|r}{A_r}-1\right|\left\|W_{q^{\rm ac}}\right\|_{L^1(Q_0)}.
\end{align*}
Using this, \eqref{eq:acmass}, and \eqref{eq:acerror}, we conclude that
\begin{align*}
\lim_{r\to 0} u_{x,r}^{R_x}=W_{q^{\rm ac}}\ \ \text{in}\ \ L^1(Q_0).
\end{align*}
This completes the  proof of Lemma \ref{lem:acpar}.
\end{proof}
Next, we   consider the jump part. We have the  following conclusion.
\begin{lemma}\label{lem:acjump}
Let $N\in\mathbb N$, $\Omega\subset\mathbb{R}^N$ be an open set,  and $u\in BV(\Omega)$.
For $|D^ju|$-a.e. $x\in J_u$, there exists
$R_x\in\mathcal O(N)$ such that, for any
sufficiently small $r\in(0,\infty)$,
$Q_r^{R_x}(x)\subset\Omega$, $|Du|(Q_r^{R_x}(x))\in (0,\infty)$, and
\begin{align}\label{eq:jumptan}
\lim_{r\to0}u_{x,r}^{R_x}=W_{q^{\rm j}}\quad\text{in }L^1(Q_0),
\end{align}
where $u^{R_x}_{x,r}$ is defined as in \eqref{eq:rescale} with $R$ replaced by $R_x.$
\end{lemma}
\begin{proof}
Let  $\eta:=|\nabla u|\mathcal L^N+|D^cu|$. By Lemma~\ref{lem:bvstruct}(iv), we conclude that $|Du|=|D^ju|+\eta$
and $|D^ju|\perp\eta.$ Extend both $|D^ju|$ and $\eta$ by zero from $\Omega$ to $\mathbb R^N$, and denote the
extensions still by $|D^ju|$ and $\eta$. Since $\eta\perp|D^ju|$,
it follows that, for $|D^ju|$-a.e. $x\in\Omega$,
\begin{align}\label{eq:jrel}
\lim_{\rho\to0^+}\frac{\eta(B(x,\rho))}{|D^ju|(B(x,\rho))}=0.
\end{align}
By Lemma \ref{lem:bvstruct}(iii), we find that
\begin{align*}
|D^ju|=|u^{+}-u^{-}| \left(\mathcal H^{N-1}\llcorner J_u\right)
\end{align*}
From this and  Lemma \ref{lem:recttan}, we infer that, for $|D^ju|$-a.e. $x\in J_u$,
\begin{align}\label{eq:jball}
\lim_{\rho\to0^+}\frac{|D^ju|(B(x,\rho))}{\rho^{N-1}}=|u^+(x)-u^-(x)|\omega_{N-1}.
\end{align}
Using Lemma \ref{lem:bvstruct}(ii) and
Lemma \ref{lem:recttan} again, we conclude that, for $|D^ju|$-a.e. $x\in J_u$,
$|D^ju|$ has approximate tangent hyperplane $T_x J_u$ with density $|u^+-u^-|$. Thus,
for $|D^ju|$-a.e. $x\in J_u$ and  any $\varphi\in C_{\rm c}(\mathbb R^N)$,
\begin{align}\label{eq:jtan}
\lim_{r\to0^+}\frac{1}{r^{N-1}}\int_{\mathbb R^N}
\varphi\left(\frac{y-x}{r}\right)\,d|D^ju|(y)=|u^+(x)-u^-(x)|
\int_{\nu_u(x)^\perp}\varphi(z)\,d\mathcal H^{N-1}(z).
\end{align}
Let $x\in J_u$  be such that  \eqref{eq:jball}, \eqref{eq:jtan}, and \eqref{eq:jrel} hold.
Interchanging $u^+(x)$ and $u^-(x)$ and reversing $\nu_u(x)$ if necessary, we may assume that
$s(x):=u^+(x)-u^-(x)>0.$ Choose $R_x\in\mathcal O(N)$ such that $R_xe_N=\nu_u(x).$ Let
\begin{align*}
A_r:=\frac{|Du|(Q_r^{R_x}(x))}{r^{N-1}}\ \ \text{and}\ \ b_r:=\fint_{Q_r^{R_x}(x)}u(y)\,dy.
\end{align*}
By the definition of $ u_{x,r}^{R_x}$, we find that, for any $z\in Q_0,$
\begin{align}\label{eq:shenying}
u_{x,r}^{R_x}(z)=\frac{u(x+rR_xz)-b_r}{A_r}.
\end{align}
Now, we show that
\begin{align}\label{eq:jumpmass}
\lim_{r\to0^+}A_r=s(x).
\end{align}
For any Borel set $E\subset\mathbb R^N$, define
\begin{align*}
\sigma_r(E):=\frac{1}{r^{N-1}}|D^ju|(x+rE)\ \ \text{and}\ \  \sigma
:=s(x)\left(\mathcal H^{N-1}\llcorner\nu_u(x)^\perp\right).
\end{align*}
Let $A:=R_xQ_0.$ Then
\begin{align*}
\sigma_r(A)=\frac{|D^ju|(Q_r^{R_x}(x))}{r^{N-1}}
\end{align*}
and
\begin{align}\label{eq:chepiao}
\sigma(\partial A)&=s(x)\int_{\nu_u(x)^\perp}
\mathbf 1_{\partial A}(z)\,d\mathcal H^{N-1}(z)=
s(x)\int_{\nu_u(x)^\perp}\mathbf 1_{R_x\partial Q_0}(z)\,d\mathcal H^{N-1}(z)\notag\\
&=s(x)\int_{e_N^\perp}\mathbf 1_{\partial Q_0}(w)\,d\mathcal H^{N-1}(w)
=s(x)\mathcal H^{N-1}\left(e_N^\perp\cap\partial Q_0\right)=0.
\end{align}
On the other hand, from \eqref{eq:jtan}, we deduce that, for any $\varphi\in C_{\rm c}(\mathbb R^N)$,
\begin{align*}
\lim_{r\to0^+}\int_{\mathbb R^N}\varphi(z)\,d\sigma_r(z)=\int_{\mathbb R^N}\varphi(z)\,d\sigma(z).
\end{align*}
Equivalently, $\sigma_r$ converges weakly to $\sigma$ as $r\to\infty.$ This, combined with
\eqref{eq:chepiao}, implies that
\begin{align}\label{eq:jpart}
\lim_{r\to0^+}\frac{|D^ju|(Q_r^{R_x}(x))}{r^{N-1}}=\lim_{r\to 0^+}\sigma_r(A)=\sigma (A)
=s(x)\mathcal H^{N-1}\left(e_N^\perp\cap Q_0\right)=s(x).
\end{align}
By \eqref{eq:jrel} and \eqref{eq:jball}, we conclude that
\begin{align*}
\lim_{\rho\to 0^+}\frac{\eta(B(x,\rho))}{\rho^{N-1}}=\lim_{\rho\to 0^+}
\frac{\eta(B(x,\rho))}{|D^ju|(B(x,\rho))}\frac{|D^ju|(B(x,\rho))}{\rho^{N-1}}=0.
\end{align*}
From this and \eqref{eq:jpart},  it follows  that
\begin{align*}
\lim_{r\to0^+}A_r=\lim_{r\to0^+}\left[\frac{|D^ju|(Q_r^{R_x}(x))}{r^{N-1}}
+\frac{\eta(Q_r^{R_x}(x))}{r^{N-1}}\right]=s(x).
\end{align*}
This proves \eqref{eq:jumpmass}.

Let
\begin{align*}
 Q_0^\pm:=Q_0\cap\left\{z\in\mathbb R^N:\ \pm z_N>0\right\}.
\end{align*}
By a change of variables and definitions of $u^+$ and $u^-,$ we find that
\begin{align*}
\int_{Q_0^+}|u(x+rR_xz)-u^+(x)|\,dz\lesssim\fint_{B^+(x,\sqrt{N}r,\nu_u(x))}|u(y)-u^+(x)|\,dy
\to0
\end{align*}
and
\begin{align*}
\int_{Q_0^-}|u(x+rR_xz)-u^-(x)|\,dz\lesssim\fint_{B^-(x,\sqrt{N}r,\nu_u(x))}|u(y)-u^-(x)|\,dy
\to0
\end{align*}
as $r\to 0^+.$ For any $z\in Q_0,$ let
\begin{align*}
w_x(z):=u^-(x)\mathbf 1_{\left\{z_N<0\right\}}+u^+(x)\mathbf 1_{\left\{z_N>0\right\}}.
\end{align*}
Then
\begin{align}\label{eq:jumplone}
\lim_{r\to0}u(x+rR_x\,\cdot)=w_x\ \ \text{in}\ \ L^1(Q_0).
\end{align}
This further implies that
\begin{align*}
\lim_{r\to0^+}b_r=\lim_{r\to0^+}\int_{Q_0}u(x+rR_xz)\,dz=\int_{Q_0}w_x(z)\,dz
 =\frac{u^+(x)+u^-(x)}{2}:=m_x,
\end{align*}
which, together with \eqref{eq:jumplone}, \eqref{eq:jumpmass}, and \eqref{eq:shenying},
implies that
\begin{align}\label{eq:jnorm}
\lim_{r\to 0^+}u_{x,r}^{R_x}=\frac{w_x-m_x}{s(x)}\ \ \text{in}\ \ L^1(Q_0).
\end{align}
Note that,  for any $z\in Q_0$ with $z_N>0$,
\begin{align*}
\frac{w_x(z)-m_x}{s(x)}&=\frac{u^+(x)-\dfrac{u^+(x)+u^-(x)}{2}}
{u^+(x)-u^-(x)}=\frac{1}{2},
\end{align*}
and, for any $z\in Q_0$ with $z_N<0,$
\begin{align*}
\frac{w_x(z)-m_x}{s(x)}&=\frac{u^-(x)-\dfrac{u^+(x)+u^-(x)}{2}}
{u^+(x)-u^-(x)}=-\frac{1}{2}.
\end{align*}
Thus, for $\mathcal L^N$-a.e.  of $Q_0$,
$$\frac{w_x-m_x}{s(x)}=W_{q^{\rm j}}.$$
From this and  \eqref{eq:jnorm}, we deduce that
\eqref{eq:jumptan} holds. This completes the  proof of Lemma \ref{lem:acjump}.
\end{proof}

Finally, we   consider the Cantor part. Since $D^cu$ is an $\mathbb R^N$-valued Radon measure,
it follows that there exists a Borel map $\theta_u^c:\ \Omega\to \mathbb S^{N-1}$ such that
\begin{align*}
D^cu=\theta_u^c\,|D^cu|.
\end{align*}
For $|D^cu|$-a.e. $x\in\Omega$, let $R_x\in \mathcal O(N)$ such that $R_xe_N
=\theta_u^c(x).$ The following lemma is a special case of \cite[Lemma~4.5]{cfi24}
with $m:=1$, $\eta:=1$ and $\xi:=\theta_u^c$.

\begin{lemma}\label{lem:cfi}
Let $N\in\mathbb N$, $\Omega\subset\mathbb{R}^N$ be an open set,
$\sigma$ be a positive Radon measure on
$\Omega$, and $u\in BV(\Omega)$. Then, for $|D^cu|$-a.e. $x\in\Omega$,
there exist a sequence $\{r_k\}_{k\in\mathbb N}\subset(0,\infty)$ with
$r_k\to0$ and a nondecreasing function
$q_x^{\rm c}:I_0\to\mathbb R$ such that
\begin{itemize}
\item [$\mathrm{(i)}$] for any $k\in\mathbb{N}$,
$\sigma(\partial Q^{R_x}_{r_k}(x))=0$;
\item [$\mathrm{(ii)}$]
\begin{align*}
\lim_{k\to\infty}\frac{|Du|(Q_{r_k}^{R_x}(x))}{r_k^N}
=\infty\ \ \text{and}\ \ \lim_{k\to\infty}\frac{|Du|(Q_{r_k}^{R_x}(x))}{r_k^{N-1}}
=0;
\end{align*}
\item [$\mathrm{(iii)}$] $|Dq_x^{\rm c}|(I_0)=1$;
\item [$\mathrm{(iv)}$]
\begin{align*}
\lim_{k\to\infty}u_{x,r_k}^{R_x}(z)=\lim_{k\to\infty}\frac{ u(x+r_kR_xz)
-\fint_{Q_{r_k}^{R_x}(x)}u(y)\,dy}{|Du|\left(Q_{r_k}^{R_x}(x)\right)r_k^{-N+1}
}=q_x^{\rm c}(z_N)
\end{align*}
strictly in $BV(Q_0).$
\end{itemize}
\end{lemma}

Now, we show Lemma \ref{lem:tangent}.
\begin{proof}[Proof of Lemma \ref{lem:tangent}]
By \eqref{eq:bvvar}, we find that the measures $|\nabla u|\mathcal L^N$, $|D^ju|,$ and $|D^cu|$
are mutually singular and $|Du|=|\nabla u|\mathcal L^N+|D^ju|+|D^cu|.$
Thus, to prove the  present lemma, it suffices to show the desired conclusions on a
full-measure set with respect to each of these three measures $|\nabla u|\mathcal{L}^N$,
$|D^ju|$, and $|D^cu|.$

We first consider the absolutely continuous part. Let $x$ be a point
such that  Lemma~\ref{lem:acpar} holds. Let all the symbols be the same as in
Lemma~\ref{lem:acpar} and $q_x:=q^{\rm ac}$.
Since $|Du|$ is finite, it follows that there are at
most countably many radii $r\in(0,\infty)$ such that
\begin{align*}
|Du|\left(\partial Q_r^{R_x}(x)\right)>0.
\end{align*}
Thus, we may  choose a sequence $\{r_k\}_{k\in\mathbb N}\subset(0,\infty)$
such that $\lim_{k\to\infty}r_k=0$ and, for any $k\in\mathbb{N},$
\begin{align*}
|Du|\left(\partial Q_{r_k}^{R_x}(x)\right)=0.
\end{align*}
This proves the desired conclusions for the absolutely continuous part.
Similarly, the jump part follows from Lemma~\ref{lem:acjump} and the Cantor part
follows from Lemma~\ref{lem:cfi} with $\sigma:=|Du|$. This completes
the proof of Lemma \ref{lem:tangent}.
\end{proof}

\subsection{A Lower Bound for Nondecreasing BV Functions of One Variable}
\label{subsec:monobv}
In this subsection, we show the following conclusion, which is essential in the proof of
Theorem \ref{thm:oneliminf}.

\begin{proposition}\label{prop:profile}
Let $\gamma\in(0,\infty)$ and $q\in BV(I_0)$.  Assume that $q$ is nondecreasing.  Then
\begin{align}\label{eq:profile}
 m_{1,\gamma}(W_q)\geq C_{N,1,\gamma}^{\mathrm{cell}}|Dq|(I_0),
\end{align}
where $I_0:=(-\frac{1}{2},\frac{1}{2})$ and, for any  $z\in Q_0,$ $W_q(z):=q(z_N).$
\end{proposition}

We first prove the following lemma.

\begin{lemma}\label{lem:replace}
Let $\gamma\in(0,\infty)$ and $\{\lambda_j\}_{j\in\mathbb{N}}\subset (0,\infty)$ satisfy $\lim_{j\to\infty}\lambda_j=\infty.$
Let $w\in L^\infty(Q_0)$ and $\{v_j\}_{j\in\mathbb N}\subset L^\infty(Q_0)$ satisfy
$0\leq w, v_j\leq 1$ and $\lim_{j\to\infty}v_j=w$ in $L^1(Q_0).$ Fix $\rho\in(0,\frac{1}{20})$.
Then there exist sequences $\{\widetilde{\lambda_j}\}_{j\in\mathbb{N}}\subset (0, \infty)$
and $\{\widehat v_j\}_{j\in\mathbb N}\subset L^\infty(Q_0)$ such that the following statements hold.
\begin{enumerate}[{\rm(i)}]
\item  For any $j\in\mathbb{N}$, $0\leq\widehat v_j\leq1$, $\widehat v_j=w$ on
$\{z\in Q_0:\operatorname{dist}(z,\partial Q_0)\leq\rho\}$,
and $\lim_{j\to\infty}\widehat v_j=w$ in $L^1(Q_0).$
\item  $\lim_{j\to\infty}\frac{\widetilde{\lambda_j}}{\lambda_j}=1$.
\item  Let $A_\rho:=\{z\in Q_0:\operatorname{dist}(z,\partial Q_0)>\rho\}$ and $C_\rho
:=\{z\in Q_0:\operatorname{dist}(z,\partial Q_0)<3\rho\}.$ Then
\begin{align}\label{eq:repenergy}
\limsup_{j\to\infty} G_{\widetilde\lambda_j,1,\gamma}\left(\widehat v_j;Q_0\right)\leq
\limsup_{j\to\infty}G_{\lambda_j,1,\gamma}\left(v_j;A_\rho\right)+\limsup_{j\to\infty}
G_{\lambda_j,1,\gamma}\left(w;C_\rho\right).
\end{align}
\end{enumerate}
\end{lemma}

\begin{proof}
If the right-hand side of \eqref{eq:repenergy} is infinite, the
conclusion is immediate by taking $\widetilde\lambda_j=\lambda_j$ and $\widehat v_j=w$.
Hence we only consider the case where the right-hand side is finite.
After discarding finitely many indices, we may assume that
\begin{align}\label{eq:repbound}
\sup_{j\in\mathbb N}G_{\lambda_j,1,\gamma}\left(v_j;A_\rho\right)<\infty .
\end{align}
For any  $j\in\mathbb{N}$, let
\begin{align*}
M_j:=\int_{A_\rho\cap C_\rho}|v_j-w|^{\frac{\gamma}{\gamma+1}}\,dz.
\end{align*}
Since \(v_j\to w\) in \(L^1(Q_0)\), it follows that $M_j\to\ 0.$
Choose a sequence $\{\eta_j\}_{j\in\mathbb N}\subset (0,\infty)$ such that
\begin{align}\label{eq:xuanze2}
\lim_{j\to\infty}\eta_j=0\ \ \text{and}\ \ \lim_{j\to\infty}\frac{\|v_j-w\|_{L^1(Q_0)}}{\eta_j}=0
\end{align}
and a sequence $\{K_j\}_{j\in\mathbb{N}}$ of integers such that
\begin{align}\label{eq:xuanze1}
\lim_{j\to\infty}K_j=\infty\ \ \text{and}\ \ \lim_{j\to\infty}\frac{
\eta_j^{-\frac{\gamma}{\gamma+1}}\lambda_j^{\frac1{\gamma+1}}M_j}{K_j}=0.
\end{align}
Moreover, for any $j\in\mathbb{N},$ let  $\widetilde{\lambda_j}:=(1+\eta_j)\lambda_j$.
Then $\lim_{j\to\infty}\frac{\widetilde\lambda_j}{\lambda_j}= 1.$
For any given $j\in\mathbb{N}$ and $\ell\in\{1,\ldots,K_j\}$, define
\begin{align*}
t_{\ell,j}:=\frac{3\rho}{2}+\frac{\ell\rho}{2(K_j+1)},\ \ E_{\ell,j}
:=\left\{z\in Q_0:\operatorname{dist}(z,\partial Q_0)>t_{\ell,j}\right\},
\end{align*}
and
\begin{align*}
v_{\ell,j}(z):=\begin{cases}
v_j(z),&z\in E_{\ell,j},\\
w(z),&z\in Q_0\setminus E_{\ell,j}.
\end{cases}
\end{align*}
Since $\frac{3\rho}{2}<t_{\ell,j}<2\rho$, it follows that $E_{\ell,j}\subset A_\rho$ and
$Q_0\setminus E_{\ell,j}\subset C_\rho$. This implies that
$v_{\ell,j}=w$ on $\{z\in Q_0:\ \operatorname{dist}(z,\partial Q_0)\leq\rho\}$ and
\begin{align}\label{eq:repconv}
\left\|v_{\ell,j}-w\right\|_{L^1(Q_0)}\leq\left\|v_j-w\right\|_{L^1(E_{\ell,j})}
\leq\left\|v_j-w\right\|_{L^1(Q_0)}.
\end{align}
Let
\begin{align*}
R_{\ell,j,1}:=\widetilde\lambda_j\int_{E_{\ell,j}}\int_{Q_0\setminus E_{\ell,j}}
\mathbf 1_{\left\{(x,y)\in E_{\ell,j}\times\left(Q_0\setminus E_{\ell,j}\right):\ x\neq y,\
\frac{|v_j(x)-w(y)|}{|x-y|^{\gamma+1}}\geq\widetilde\lambda_j\right\}}
|x-y|^{\gamma-N}\,dy\,dx
\end{align*}
and
\begin{align*}
R_{\ell,j,2}:=\widetilde\lambda_j\int_{Q_0\setminus E_{\ell,j}}\int_{E_{\ell,j}}
\mathbf 1_{\left\{(x,y)\in\left(Q_0\setminus E_{\ell,j}\right)\times E_{\ell,j}:\ x\neq y,\
\frac{|w(x)-v_j(y)|}{|x-y|^{\gamma+1}}\geq\widetilde\lambda_j\right\}}
|x-y|^{\gamma-N}\,dy\,dx.
\end{align*}
Then we  have
\begin{align}\label{eq:repdecomp}
G_{\widetilde\lambda_j,1,\gamma}(v_{\ell,j};Q_0)
\leq(1+\eta_j)G_{\lambda_j,1,\gamma}(v_j;A_\rho)+(1+\eta_j)G_{\lambda_j,1,\gamma}(w;C_\rho)
+R_{\ell,j,1}+R_{\ell,j,2}.
\end{align}
Next, we show
\begin{align}\label{eq:nangao}
\lim_{j\to\infty}\frac{1}{K_j}\sum_{\ell=1}^{K_{j}}\left(R_{\ell,j,1}+R_{\ell,j,2}\right)=0.
\end{align}
For any $j\in\mathbb{N}$ and  $\ell\in\{1,\ldots,K_j\}$, define
\begin{align*}
\Gamma_{\ell,j,1}&:=\left\{(x,y)\in E_{\ell,j}\times
\left(Q_0\setminus E_{\ell,j}\right):\ x\neq y,\
\frac{|v_j(x)-w(y)|}{|x-y|^{\gamma+1}}\geq\widetilde\lambda_j\right\}.
\end{align*}
Since $\lambda_j\to\infty$, after discarding finitely many indices, we
may assume that, for any $j\in\mathbb{N},$ $ \lambda_j^{-\frac{1}{\gamma+1}}
<\frac{\rho}{4}.$ From this  and  the facts that  $0\leq v_j,w\leq1$  and
$\widetilde\lambda_j\geq\lambda_j$, we deduce that
\begin{align}\label{eq:naoziteng}
\bigcup_{\ell=1}^{K_j}\Gamma_{\ell,j,1}&\subset\left\{(x,y)\in Q_0\times Q_0:\ x\neq y,\
\frac{1}{|x-y|^{\gamma+1}}\geq\widetilde\lambda_j\right\}\notag\\
&\subset\left\{(x,y)\in Q_0\times Q_0:|x-y|\leq\frac{\rho}{4}\right\}.
\end{align}
We claim that
\begin{align}\label{eq:repregion}
\bigcup_{\ell=1}^{K_j}\Gamma_{\ell,j,1}\subset\left(A_\rho\cap C_\rho\right)
\times\left(A_\rho\cap C_\rho\right).
\end{align}
Indeed, let $(x,y)\in\Gamma_{\ell,j,1}$ for some $\ell\in\{1,\ldots,K_j\}$. Then
\begin{align*}
\operatorname{dist}(x,\partial Q_0)>t_{\ell,j}\geq\operatorname{dist}(y,\partial Q_0).
\end{align*}
Using this, \eqref{eq:naoziteng}, and the fact that $\frac{3\rho}{2} <t_{\ell,j}
<2\rho,$ we conclude that
\begin{align*}
\rho<t_{\ell,j}<\operatorname{dist}(x,\partial Q_0)\leq\operatorname{dist}(y,\partial Q_0)+|x-y|
\leq t_{\ell,j}+\frac{\rho}{4}<\frac{9\rho}{4}<3\rho
\end{align*}
and
\begin{align*}
\rho<t_{\ell,j}-\frac{\rho}{4}<\operatorname{dist}(x,\partial Q_0)-|x-y|
\leq\operatorname{dist}(y,\partial Q_0)\leq t_{\ell,j}<3\rho.
\end{align*}
Hence $(x,y)\in(A_\rho\cap C_\rho)\times(A_\rho\cap C_\rho)$, which proves \eqref{eq:repregion}.
For any $(x,y)\in Q_0\times Q_0$,  let
\begin{align*}
N_{j,1}(x,y):=\sum_{\ell=1}^{K_j} \mathbf 1_{\{x\in E_{\ell,j},\,y\in Q_0\setminus E_{\ell,j}\}}(x,y).
\end{align*}
Then, by \eqref{eq:repregion}, we find that
\begin{align*}
\frac{1}{K_j}\sum_{\ell=1}^{K_j}R_{\ell,j,1}&=\frac{\widetilde\lambda_j}{K_j}
\int_{A_\rho\cap C_\rho}\int_{A_\rho\cap C_\rho}N_{j,1}(x,y)
\mathbf 1_{\left\{x,y\in A_\rho\cap C_\rho:\ x\neq y,\
\frac{|v_j(x)-w(y)|}{|x-y|^{\gamma+1}}\geq\widetilde\lambda_j\right\}}
|x-y|^{\gamma-N}\,dy\,dx,
\end{align*}
which further implies that
\begin{align}\label{eq:repij}
\frac{1}{K_j}\sum_{\ell=1}^{K_j}R_{\ell,j,1}\leq I_{j,1}+I_{j,2},
\end{align}
where
\begin{align*}
I_{j,1}:=\frac{\widetilde\lambda_j}{K_j}\int_{A_\rho\cap C_\rho}\int_{A_\rho\cap C_\rho}
N_{j,1}(x,y)\mathbf 1_{\left\{x,y\in A_\rho\cap C_\rho:\ x\neq y,\
\frac{|v_j(x)-v_j(y)|}{|x-y|^{\gamma+1}}\geq\lambda_j\right\}}
|x-y|^{\gamma-N}\,dy\,dx
\end{align*}
and
\begin{align*}
I_{j,2}:=\frac{\widetilde\lambda_j}{K_j}\int_{A_\rho\cap C_\rho}\int_{A_\rho\cap C_\rho}
N_{j,1}(x,y)\mathbf 1_{\left\{x,y\in A_\rho\cap C_\rho:\ x\neq y,\
\frac{|v_j(y)-w(y)|}{|x-y|^{\gamma+1}}\geq\eta_j\lambda_j\right\}}
|x-y|^{\gamma-N}\,dx\,dy.
\end{align*}
Observe that, for any  $x,y\in Q_0$,
\begin{align*}
&\left\{\ell\in\{1,\ldots,K_j\}:x\in E_{\ell,j},\
y\in Q_0\setminus E_{\ell,j}\right\}\\
&\quad=\left\{\ell\in\{1,\ldots,K_j\}:
\operatorname{dist}(y,\partial Q_0)\leq t_{\ell,j}
<\operatorname{dist}(x,\partial Q_0)\right\}.
\end{align*}
This, combined with the fact that  $ t_{\ell+1,j}-t_{\ell,j}=\frac{\rho}{2(K_j+1)}$,
implies that, for any $x,y\in Q_0,$
\begin{align}\label{eq:repcount}
N_{j,1}(x,y)&\leq1+\frac{\left(\operatorname{dist}(x,\partial Q_0)-
\operatorname{dist}(y,\partial Q_0)\right)_+}{\frac{\rho}{2(K_j+1)}}\leq
1+\frac{2(K_j+1)|x-y|}{\rho}.
\end{align}
Since $0\leq v_j\leq 1$, it follows that, for any $x,y\in Q_0$
satisfying $|v_j(x)-v_j(y)|\geq \lambda_{j}|x-y|^{\gamma+1},$ $|x-y|\leq \lambda_j^{-\frac{1}{\gamma+1}}.$
Using this and \eqref{eq:repcount}, we conclude that
\begin{align*}
I_{j,1}&=\frac{(1+\eta_j)\lambda_j}{K_j}\int_{A_\rho\cap C_\rho}\int_{A_\rho\cap C_\rho}
N_{j,1}(x,y)\mathbf 1_{\left\{x,y\in A_\rho\cap C_\rho:\ x\neq y,\
\frac{|v_j(x)-v_j(y)|}{|x-y|^{\gamma+1}}\geq\lambda_j\right\}}
|x-y|^{\gamma-N}\,dy\,dx\\
&\lesssim(1+\eta_j)\lambda_j\int_{A_\rho\cap C_\rho}
\int_{A_\rho\cap C_\rho}\left(\frac{1}{K_j}+\frac{|x-y|}{\rho}\right)\mathbf 1_{\left\{
x,y\in A_\rho\cap C_\rho:\ x\neq y,\
\frac{|v_j(x)-v_j(y)|}{|x-y|^{\gamma+1}}\geq\lambda_j\right\}}
|x-y|^{\gamma-N}\,dy\,dx\\
&\leq(1+\eta_j)\lambda_j\int_{A_\rho\cap C_\rho}
\int_{A_\rho\cap C_\rho}\left(\frac{1}{K_j}+\frac{\lambda_j^{-\frac{1}{\gamma+1}}}{\rho}\right)\mathbf 1_{\left\{
x,y\in A_\rho\cap C_\rho:\ x\neq y,\
\frac{|v_j(x)-v_j(y)|}{|x-y|^{\gamma+1}}\geq\lambda_j\right\}}
|x-y|^{\gamma-N}\,dy\,dx\\
&\leq(1+\eta_j)\left(\frac{1}{K_j}+\frac{\lambda_j^{-\frac{1}{\gamma+1}}}{\rho}\right)
G_{\lambda_j,1,\gamma}\left(v_j;A_\rho\right).
\end{align*}
Letting $j\to\infty$ and using \eqref{eq:repbound}, we have
\begin{align}\label{eq:repizero}
\lim_{j\to\infty}I_{j,1}=0.
\end{align}
Similarly, we also have
\begin{align*}
I_{j,2}\lesssim(1+\eta_j)\frac{\eta_j^{-\frac{\gamma}{\gamma+1}}\lambda_j^{\frac{1}{\gamma+1}}M_j
}{K_j}+(1+\eta_j)\frac{\|v_j-w\|_{L^1(Q_0)}}{\rho\eta_j}.
\end{align*}
Letting $j\to\infty$ and using \eqref{eq:xuanze1} and \eqref{eq:xuanze2}, we conclude that
$\lim_{j\to\infty}I_{j,2}=0.$ This, together with \eqref{eq:repij} and  \eqref{eq:repizero}, implies  that
\begin{align*}
\lim_{j\to\infty}\frac{1}{K_j}\sum_{\ell=1}^{K_j}R_{\ell,j,1}=0.
\end{align*}
Similarly, we also have
\begin{align*}
\lim_{j\to\infty}\frac{1}{K_j}\sum_{\ell=1}^{K_j}R_{\ell,j,2}=0
\end{align*}
and hence \eqref{eq:nangao} holds.

For any given  $j\in\mathbb{N}$, choose $\ell(j)\in\{1,\ldots,K_j\}$ such that
\begin{align*}
R_{\ell(j),j,1}+R_{\ell(j),j,2}\leq\frac{1}{K_j}\sum_{\ell=1}^{K_j}\left(
R_{\ell,j,1}+R_{\ell,j,2}\right),
\end{align*}
and define $\widehat v_j:=v_{\ell(j),j}.$ From \eqref{eq:nangao}, we infer that
\begin{align*}
\lim_{j\to\infty}\left(R_{\ell(j),j,1}+R_{\ell(j),j,2}\right)=0,
\end{align*}
which, together with \eqref{eq:repdecomp}, implies that \eqref{eq:repenergy} holds.
Moreover, using \eqref{eq:repconv}, we have
\begin{align*}
\lim_{j\to\infty}\left\|\widehat v_j-w\right\|_{L^1(Q_0)}=0.
\end{align*}
The other assertions of the present lemma holds  naturally
by the construction of $v_{\ell,j}$. This completes  the proof of Lemma \ref{lem:replace}.
\end{proof}

The following lemma is just  \cite[Proposition~2.3(i)]{bsvy24}.

\begin{lemma}\label{lem:segment}
Let $f\in C_{\rm c}(\mathbb R^N)$, $\gamma\in(0,\infty)$, and
\begin{align*}
\mathcal E(f,\gamma):=\left\{(x,y)\in\mathbb R^N\times\mathbb R^N:\ x\neq y,\
\int_{[x,y]}|f(s)|\,ds>|x-y|^{\gamma+1}\right\},
\end{align*}
where $[x,y]$ denotes the closed segment joining $x$ and $y$. Then there exists a positive constant
$C_{N,\gamma}$, depending only on $N$ and $\gamma$, such that
\begin{align*}
\iint_{\mathcal E(f,\gamma)}|x-y|^{\gamma-N}\,dx\,dy\leq C_{N,\gamma}\|f\|_{L^1(\mathbb R^N)}.
\end{align*}
\end{lemma}

For any  measurable set $E\subset\mathbb R^N$ and $r\in (0,\infty)$, let
\begin{align*}
 E^{(r)}:=\{x\in\mathbb R^N:\operatorname{dist}(x,E)<r\}.
\end{align*}

\begin{lemma}\label{lem:localbv}
Let $N\in\mathbb N$, $\gamma\in(0,\infty)$, $v\in BV(\mathbb R^N)$, and $E\subset\mathbb R^N$ be a bounded Borel set.  Then
there exists a  positive constant $C_{N,\gamma}$, depending only on $N$ and $\gamma$, such that, for any $r,\lambda\in (0,\infty)$,
\begin{align}\label{eq:localbv}
\lambda\int_E\int_{B(x,r)}\mathbf 1_{\left\{x\in E,\ y\in B(x,r):\ x\neq y,\
\frac{|v(x)-v(y)|}{|x-y|^{\gamma+1}}\geq\lambda\right\}}
|x-y|^{\gamma-N}\,dy\,dx
\leq C_{N,\gamma}|Dv|(E^{(3r)}).
\end{align}
\end{lemma}

\begin{proof}
Assume first that $v\in C^\infty(\mathbb R^N)$. Fix $\varepsilon\in(0,1)$.
Choose $\chi\in C_{\rm c}^\infty(E^{(2r)})$ such that $0\leq\chi\leq1$ and $\chi\equiv 1$ on $E^{(r)}$.
Let $x\in E$ and $y\in B(x,r)$. Then the whole segment $[x,y]$ is contained in $E^{(r)}$.
This implies that  $\chi=1$ on $[x,y]$ and hence
\begin{align*}
|v(x)-v(y)|\leq\int_{[x,y]}|\nabla v(s)|\,ds=\int_{[x,y]}\chi(s)|\nabla v(s)|\,ds.
\end{align*}
Thus, if $|v(x)-v(y)|\geq\lambda|x-y|^{\gamma+1},$ then
\begin{align*}
\int_{[x,y]}\frac{\chi(s)|\nabla v(s)|}{(1-\varepsilon)\lambda}\,ds> |x-y|^{\gamma+1}.
\end{align*}
Using this and applying Lemma~\ref{lem:segment} with
$f:=\frac{\chi|\nabla v|}{(1-\varepsilon)\lambda},$ we conclude  that
\begin{align*}
&\int_E\int_{B(x,r)}\mathbf 1_{\left\{x\in E,\ y\in B(x,r):\ x\neq y,\
\frac{|v(x)-v(y)|}{|x-y|^{\gamma+1}}\geq\lambda\right\}}
|x-y|^{\gamma-N}\,dy\,dx\\
&\quad\leq\iint_{\mathcal E(f,\gamma)}|x-y|^{\gamma-N}\,dx\,dy\lesssim\frac{1}{(1-\varepsilon)\lambda}
\int_{\mathbb R^N}\chi|\nabla v|\,dx\leq\frac{1}{(1-\varepsilon)\lambda}
\int_{E^{(2r)}}|\nabla v|\,dx.
\end{align*}
Letting $\varepsilon \to0$ in the above inequality, we then have
\begin{align}\label{eq:lbvsmooth}
\lambda\int_E\int_{B(x,r)}\mathbf 1_{\left\{x\in E,\ y\in B(x,r):\ x\neq y,\
\frac{|v(x)-v(y)|}{|x-y|^{\gamma+1}}\geq\lambda\right\}}
|x-y|^{\gamma-N}\,dy\,dx
\lesssim\int_{E^{(2r)}}|\nabla v|\,dx.
\end{align}
Now, let $v\in BV(\mathbb R^N)$. Choose a standard mollifier $\rho\in C_{\rm c}^\infty(B({\bf0},1))$ such that
$\rho\geq0$ and $\int_{\mathbb R^N}\rho\,dx=1.$ For any $\delta\in (0,\infty)$ and $z\in\mathbb{R}^N,$ let
$\rho_\delta(z):=\delta^{-N}\rho(\frac{z}{\delta}).$
Choose a sequence $\delta_k\downarrow0$ with $\delta_k<r$, and
let $v_k:=v*\rho_{\delta_k}.$ After passing to a subsequence, we may assume that, for almost every $x\in\mathbb{R}^N, $
$\lim_{k\to\infty}v_k(x)=v(x).$ Fix $\varepsilon\in(0,1)$. By \eqref{eq:lbvsmooth} and the Fatou lemma, we find that
\begin{align}\label{eq:lbvfatou}
&\lambda\int_E\int_{B(x,r)}\mathbf 1_{\left\{x\in E,\ y\in B(x,r):\ x\neq y,\
\frac{|v(x)-v(y)|}{|x-y|^{\gamma+1}}\geq\lambda\right\}}
|x-y|^{\gamma-N}\,dy\,dx\notag\\
&\quad\leq\lambda\int_E\int_{B(x,r)}\liminf_{k\to\infty}\mathbf 1_{\left\{
x\in E,\ y\in B(x,r):\ x\neq y,\
\frac{|v_k(x)-v_k(y)|}{|x-y|^{\gamma+1}}\geq(1-\varepsilon)\lambda\right\}}
|x-y|^{\gamma-N}\,dy\,dx \notag\\
&\quad\lesssim\frac{1}{1-\varepsilon}\liminf_{k\to\infty}\int_{E^{(2r)}}|\nabla v_k|\,dx.
\end{align}
On the other hand, from the fact that $ \nabla v_k=\rho_{\delta_k}*Dv$ and Fubini's theorem,
we deduce that, for any $k\in\mathbb{N},$
\begin{align*}
\int_{E^{(2r)}}|\nabla v_k(x)|\,dx&\leq\int_{\mathbb R^N}\left[
\int_{E^{(2r)}}\rho_{\delta_k}(x-y)\,dx\right]d|Dv|(y)\\
&=\int_{\mathbb R^N}\rho_{\delta_k}(z)\left[\int_{\mathbb R^N}
\mathbf 1_{E^{(2r)}}(y+z)\,d|Dv|(y)\right]dz\\
&\leq\int_{\mathbb R^N}\rho_{\delta_k}(z)|Dv|(E^{(3r)})dz=|Dv|(E^{(3r)}).
\end{align*}
By this and  \eqref{eq:lbvfatou}, we conclude that \eqref{eq:localbv} holds. This completes
the proof of  Lemma \ref{lem:localbv}.
\end{proof}

\begin{lemma}\label{lem:collar}
Let $\gamma\in(0,\infty)$ and $q\in BV(I_0)$ be nondecreasing. Assume that  $q(-\frac{1}{2}+)=0$
and $q(\frac{1}{2}-)=1$, where $q(-\frac{1}{2}+):=
\lim_{t\to -\frac{1}{2}^+}q(t)$ and $q(\frac{1}{2}-):=\lim_{t\to\frac{1}{2}^-}q(t).$
For any $\rho\in(0,\infty)$, let $B_\rho:=\{z\in Q_0:\operatorname{dist}(z,\partial Q_0)<\rho\}$.
Then
\begin{align}\label{eq:collar}
\lim_{\rho\to 0^+}\limsup_{\lambda\to\infty}
\lambda\int_{B_\rho}\int_{Q_0}
\mathbf 1_{\left\{(x,y)\in B_\rho\times Q_0:\ x\neq y,\
\frac{|W_q(x)-W_q(y)|}{|x-y|^{\gamma+1}}\geq\lambda\right\}}
|x-y|^{\gamma-N}\,dy\,dx=0.
\end{align}
\end{lemma}

\begin{proof}
Because $q$ is  nondecreasing, $q(-\frac{1}{2}+)=0$ and $q(\frac{1}{2}-)=1$,
we may assume, after changing $q$ on a countable set if necessary, that
$0\leq q\leq 1$ on $I_0.$ Extend $q$ constantly outside $I_0$ by setting, for any $t\in\mathbb R$,
\begin{align*}
\widetilde q(t):=\begin{cases}
0,&t\in(-\infty,-\frac{1}{2}],\\
q(t),&t\in(-\frac{1}{2},\frac{1}{2}),\\
1,&t\in[\frac{1}{2},\infty).
\end{cases}
\end{align*}
Then $\widetilde{q}$ is continuous at $-\frac{1}{2}$ and $\frac{1}{2}$, which implies that
\begin{align}\label{eq:colend}
D\widetilde q\left(\left\{-\frac{1}{2}\right\}\right)
=D\widetilde q\left(\left\{\frac{1}{2}\right\}\right)=0.
\end{align}
Choose $\theta\in C_{\rm c}^\infty(\mathbb R)$ such that $\theta\equiv 1$   on a
neighborhood of $[-\frac{1}{2},\frac{1}{2}].$ If $N:=1$, let $\zeta\equiv1$ and, if $N\ge2$,
choose $\zeta\in C_{\rm c}^\infty(\mathbb R^{N-1})$ such that $\zeta\equiv 1$   on a neighborhood
of $[-\frac{1}{2},\frac{1}{2}]^{N-1}.$ For any $z:=(z',z_N)\in\mathbb{R}^N,$ define
\begin{align*}
\widetilde W_q(z',z_N):=\zeta(z')\theta(z_N)\widetilde q(z_N).
\end{align*}
Then $\widetilde W_q\in BV(\mathbb R^N)$.
Moreover, on a neighborhood $V$ of $\overline{Q}_0$, we find that, for any $(z',z_N)\in V$,
$\widetilde W_q(z',z_N)=\widetilde q(z_N).$
This further implies that, for any $i\in\{1,\ldots, N-1\}$ and
$\varphi\in C_{\rm c}^\infty(V)$,
\begin{align*}
\left\langle D_i\widetilde W_q,\varphi\right\rangle&=-\int_V\widetilde W_q(z)
\frac{\partial\varphi}{\partial z_i}(z)\,dz
=-\int_V \widetilde{q}(z_N)\frac{\partial\varphi}{\partial z_i}(z)\,dz=0
\end{align*}
and
\begin{align*}
\left\langle D_N\widetilde W_q,\varphi\right\rangle=-\int_V\widetilde W_q(z',z_N)
\frac{\partial\varphi}{\partial z_N}(z',z_N)\,dz'\,dz_N=
\int_{\mathbb R}\int_{\mathbb R^{N-1}}\varphi(z',t)\,dz'dD\widetilde q(t).
\end{align*}
Using this, we conclude that, for any $i\in\{1,\ldots, N-1\}$,
\begin{align}\label{eq:dou2}
(D_i\widetilde W_q)\llcorner V=0\ \ \text{and}\ \ (D_N\widetilde W_q)\llcorner V
=(\mathcal L^{N-1}\otimes D\widetilde q)\llcorner V,
\end{align}
where $\mathcal L^{N-1}\otimes D\widetilde q$ denotes the product measure of $\mathcal L^{N-1}$ and $D\widetilde q.$
Consequently, for any Borel sets $A'\subset\mathbb R^{N-1}$ and $B\subset\mathbb R$ satisfying
$A'\times B\subset V$, we have
\begin{align*}
|D\widetilde W_q|(A'\times B)=\mathcal L^{N-1}(A')\,|D\widetilde q|(B).
\end{align*}
From this and \eqref{eq:colend},  it follows that
\begin{align}\label{eq:colzero}
|D\widetilde W_q|(\partial Q_0)=0.
\end{align}
Applying Lemma \ref{lem:localbv} and using the fact that $0\leq W_q\leq 1$, we
conclude that, for any $\rho,\lambda\in (0,\infty),$
\begin{align*}
&\lambda\int_{B_\rho}\int_{Q_0}\mathbf 1_{\left\{(x,y)\in B_\rho\times Q_0:\ x\neq y,\
\frac{|W_q(x)-W_q(y)|}{|x-y|^{\gamma+1}}\geq\lambda\right\}}|x-y|^{\gamma-N}\,dy\,dx\notag\\
&\quad\leq\lambda\int_{B_\rho}\int_{B(x,R_\lambda)}
\mathbf 1_{\left\{x\in B_\rho,\ y\in B(x,R_\lambda):\ x\neq y,\
\frac{|\widetilde W_q(x)-\widetilde W_q(y)|}{|x-y|^{\gamma+1}}\geq\lambda\right\}}
|x-y|^{\gamma-N}\,dy\,dx\lesssim|D\widetilde W_q|\left(B_\rho^{(3R_\lambda)}\right),
\end{align*}
where $R_\lambda:=\lambda^{-\frac{1}{\gamma+1}}$
and $B_\rho^{(3R_\lambda)}:= \{x\in\mathbb{R}^N: \operatorname{dist}(x, B_\rho)<3R_\lambda\}.$
Letting $\lambda\to\infty$ and $\rho\to 0$ and using \eqref{eq:colzero}, we find that
\eqref{eq:collar} holds. This completes the proof of Lemma \ref{lem:collar}.
\end{proof}

\begin{lemma}\label{lem:skeleton}
Let $\gamma\in(0,\infty)$, $L\in\mathbb N$, and $\mathscr Q:=\{Q_1,\ldots,Q_L\}$ be a finite family of pairwise disjoint open cubes contained in $Q_0$
and define
\begin{align*}
S_{\mathscr Q}:=\overline{Q_0\setminus\left(\bigcup_{i=1}^LQ_i\right)}.
\end{align*}
Let $ U\in BV(Q_0)\cap L^\infty(Q_0)$ and assume that $U$ admits an extension
$\widetilde U\in BV(\mathbb R^N)\cap L^\infty(\mathbb R^N)$
such that $\widetilde U=U$ almost everywhere in $Q_0$ and
\begin{align}\label{eq:skelzero}
|D\widetilde U|(S_{\mathscr Q})=0.
\end{align}
For $\lambda\in(0,\infty)$, define
\begin{align*}
 G_{\lambda,1,\gamma}^{\rm cross}(U;\mathscr Q):=\lambda\sum_{\substack{1\leq i,j\leq L\\i\neq j}}
\int_{Q_i}\int_{Q_j}\mathbf 1_{\left\{(x,y)\in Q_i\times Q_j:\ x\neq y,\
\frac{|U(x)-U(y)|}{|x-y|^{\gamma+1}}\geq\lambda\right\}}
|x-y|^{\gamma-N}\,dy\,dx.
\end{align*}
Then
\begin{align}\label{eq:skeleton}
\lim_{\lambda\to\infty}G_{\lambda,1,\gamma}^{\rm cross}(U;\mathscr Q)=0.
\end{align}
\end{lemma}

\begin{proof}
Without loss of generality, we may assume that, for any $x\in Q_0,$ $|U(x)|\leq \|U\|_{L^\infty(Q_0)}$.
If $\|U\|_{L^\infty(Q_0)}=0$, then \eqref{eq:skeleton} holds automatically.
Assume that $\|U\|_{L^\infty(Q_0)}>0$.
Let $i,j\in\{1,\ldots,L\}$ with $i\neq j$, and fix points $x\in Q_i$ and $y\in Q_j$ such that
$|U(x)-U(y)|\geq\lambda|x-y|^{\gamma+1}.$ Then we have
\begin{align*}
\lambda|x-y|^{\gamma+1}\leq|U(x)-U(y)|\leq2\|U\|_{L^\infty(Q_0)}
\end{align*}
and hence
$|x-y|\leq R_{\lambda}$, where $R_\lambda:=\left[\frac{2\|U\|_{L^\infty(Q_0)}}{\lambda}\right]^{\frac{1}{\gamma+1}}.$
Since $Q_i$ and $Q_j$ are disjoint, it follows that there exists a point $z\in[x,y]\cap S_{\mathscr Q}.$
This implies that
\begin{align*}
\operatorname{dist}(x,S_{\mathscr Q})\leq |x-z|\leq |x-y|\leq R_\lambda.
\end{align*}
Consequently, $x\in S_{\mathscr Q}^{(R_\lambda)}$ and $y\in B(x,R_\lambda).$
Using this and applying Lemma~\ref{lem:localbv} with $E:=S_{\mathscr Q}^{(R_\lambda)}$ and $r:=R_\lambda$,
we conclude that
\begin{align*}
G_{\lambda,1,\gamma}^{\rm cross}(U;\mathscr Q)&\leq\lambda\!\int_{S_{\mathscr Q}^{(R_\lambda)}}
\!\int_{B(x,R_\lambda)}\mathbf 1_{\left\{x\in S_{\mathscr Q}^{(R_\lambda)},\ y\in B(x,R_\lambda):\
x\neq y,\
\frac{|\widetilde U(x)-\widetilde U(y)|}{|x-y|^{\gamma+1}}\geq\lambda\right\}}
|x-y|^{\gamma-N}\,dy\,dx\\
&\quad\lesssim|D\widetilde U|\left(S_{\mathscr Q}^{(4R_\lambda)}\right).
\end{align*}
Letting $\lambda\to \infty$ and using \eqref{eq:skelzero}, we find  that \eqref{eq:skeleton}
holds. This  completes the proof of Lemma \ref{lem:skeleton}.
\end{proof}

Now, we show Proposition \ref{prop:profile}.

\begin{proof}[Proof of Proposition \ref{prop:profile}]
Since $q$  is nondecreasing, it follows that $|Dq|(I_0)=q(\frac{1}{2}-)-q(-\frac{1}{2}+)$, where
$q(-\frac{1}{2}+):=\lim_{t\to -\frac{1}{2}^+}q(t)$ and $q(\frac{1}{2}-):=\lim_{t\to\frac{1}{2}^-}q(t).$
If $|Dq|(I_0)=0,$ then \eqref{eq:profile} holds  automatically. Let $|Dq|(I_0)\in (0,\infty).$ For any $t\in I_0,$ define
\begin{align*}
\widetilde q(t):=\frac{q(t)-q(-\frac{1}{2}+)}{|Dq|(I_0)}.
\end{align*}
Then $0\leq\widetilde q\leq1$ on $I_0$, $\widetilde q(-\frac{1}{2}+)=0$, and $\widetilde q(\frac{1}{2}-)=1.$
By \eqref{eq:mpscale}, we find that
\begin{align*}
 m_{1,\gamma}(W_q)=|Dq|(I_0)m_{1,\gamma}(W_{\widetilde q}).
\end{align*}
Thus, to prove \eqref{eq:profile}, it suffices to show that
\begin{align}\label{eq:profnorm}
 m_{1,\gamma}(W_{\widetilde q})\geq C_{N,1,\gamma}^{\mathrm{cell}}.
\end{align}
Fix $m\in\mathbb N$.  For any $k\in\{0,\ldots,m\}$,  let $a_k:=-\frac12+\frac{k}{m}.$
Define the function $q_m:I_0\to\mathbb R$ as follows.  For any
$k\in\{0,\ldots,m-1\}$ and any  $t\in(a_k,a_{k+1})$, let
\begin{align*}
q_m(t):=a_k+\frac{1}{m}\widetilde q\left(m(t-a_k)-\frac12\right).
\end{align*}
Moreover, for any $k\in\{1,\ldots, m-1\},$ let $q_m(a_k):=a_k.$ Then, we have, for any
$k\in\left\{0,\ldots,m-1\right\}$, $q_m(a_k+)=a_k$ and $q_m(a_{k+1}-)=a_k+\frac{1}{m}=a_{k+1}.$
This implies that, for any $k\in\{1,\ldots,m-1\},$
$Dq_m\left(\left\{a_k\right\}\right)=0$ and
\begin{align*}
|Dq_m|\left((a_k,a_{k+1})\right)=q_m(a_{k+1}-)-q_m(a_k+)=\frac{1}{m}.
\end{align*}
Consequently,
\begin{align*}
|Dq_m|(I_0)=\sum_{k=0}^{m-1}|Dq_m|\left((a_k,a_{k+1})\right)=1.
\end{align*}
Furthermore, for any $k\in \{0,\ldots,m-1\}$ and $t\in(a_k,a_{k+1})$,
\begin{align*}
|q_m(t)-t|=\frac{1}{m}\left|\widetilde q\left(m(t-a_k)-\frac12\right)-m(t-a_k)
\right|\leq\frac{1}{m},
\end{align*}
where we used the facts that $0\leq\widetilde{q}\leq 1$ and $m(t-a_k)\in (0,1)$ in the  last step.
The same estimate holds at the internal partition points.  Consequently,
\begin{align*}
\sup_{t\in I_0}|q_m(t)-t|\leq\frac{1}{m},
\end{align*}
and hence $W_{q_m}\to \ell$ in $L^1(Q_0).$ From this and Lemma \ref{lem:mplsc}, we infer that
\begin{align*}
C_{N,1,\gamma}^{\mathrm{cell}}=m_{1,\gamma}(\ell)\leq\liminf_{m\to\infty}m_{1,\gamma}(W_{q_m}).
\end{align*}
Using this, we conclude that, to prove \eqref{eq:profnorm}, it suffices to show that, for any $m\in\mathbb{N},$
\begin{align}\label{eq:zhongji}
m_{1,\gamma}(W_{q_m})\leq m_{1,\gamma}(W_{\widetilde q}).
\end{align}

We may assume that $m_{1,\gamma}(W_{\widetilde q})<\infty$, otherwise the
required comparison is immediate. Fix $\varepsilon\in (0,\infty)$.  By the definition of $m_{1,\gamma}(W_{\widetilde q})$,
we conclude that there
exist sequences $\{\lambda_j\}_{j\in\mathbb{N}}\subset (0,\infty)$ and
$\{\widetilde{v_j}\}_{j\in\mathbb{N}}\subset L^1(Q_0)$ such that $\lim_{j\to\infty}\lambda_j=\infty$,
$\lim_{j\to\infty}\widetilde{v_j}=W_{\widetilde{q}}$ in $L^1(Q_0)$, and
\begin{align}\label{eq:zuo}
\liminf_{j\to\infty}G_{\lambda_j,1,\gamma}(\widetilde{v_j};Q_0)
\leq m_{1,\gamma}(W_{\widetilde q})+\varepsilon.
\end{align}
For any $s\in\mathbb{R}$ and $j\in\mathbb{N},$
let  $ P(s):=\max\left\{0,\min\left\{1,s\right\}\right\}$ and $v_j:=P\circ \widetilde{v_j}.$
Then $0\leq P\leq 1$ and, for any $s_1,s_2\in\mathbb{R},$ $|P(s_1)-P(s_2)|\leq |s_1-s_2|.$ From
this and the fact that $0\leq W_{\widetilde{q}}\leq 1$, it follows  that, for any $j\in\mathbb{N}$,
\begin{align*}
\left\|v_j-W_{\widetilde q}\right\|_{L^1(Q_0)}=
\left\|P\circ \widetilde{v_j}-P\circ W_{\widetilde q}\right\|_{L^1(Q_0)}
\leq\left\|\widetilde{v_j}-W_{\widetilde q}\right\|_{L^1(Q_0)}.
\end{align*}
Thus, $\lim_{j\to\infty} v_j=W_{\widetilde{q}}$ in $L^1(Q_0)$.
Note that, for any $x,y\in Q_0$,
\begin{align*}
\left|P(\widetilde{v_j}(x))-P(\widetilde{v_j}(y))\right|\leq\left|\widetilde{v_j}(x)-\widetilde{v_j}(y)\right|.
\end{align*}
This implies that, for any $j\in\mathbb{N}$ and $\lambda\in (0,\infty)$,
\begin{align*}
G_{\lambda,1,\gamma}(v_j;Q_0)\leq G_{\lambda,1,\gamma}(\widetilde{v_j};Q_0).
\end{align*}
Using this and \eqref{eq:zuo}, we find that
\begin{align*}
\liminf_{j\to\infty}G_{\lambda_j,1,\gamma}(v_j;Q_0)
\leq \liminf_{j\to\infty}G_{\lambda_j,1,\gamma}(\widetilde{v_j};Q_0)
\leq m_{1,\gamma}(W_{\widetilde q})+\varepsilon.
\end{align*}
Passing to a subsequence realizing the limit inferior and relabelling, we may assume that
\begin{align}\label{eq:profnear}
\lim_{j\to\infty}G_{\lambda_j,1,\gamma}(v_j;Q_0)\leq m_{1,\gamma}(W_{\widetilde q})+\varepsilon.
\end{align}
Moreover, from the definition of $P$, we deduce that $0\leq v_j\leq 1$ for any $j\in\mathbb{N}.$
Let $\rho,$ $A_\rho$, and $C_\rho$  be as in Lemma~\ref{lem:replace}.
Apply Lemma~\ref{lem:replace} with $w:=W_{\widetilde q}$  and $\{v_j\}_{j\in\mathbb{N}}$,
we conclude that there exist sequences $\{\widetilde{\lambda_j}\}_{j\in\mathbb{N}}\subset (0,\infty)$
and $\{\widehat{v_j}\}_{j\in\mathbb{N}}\subset L^\infty (Q_0)$
satisfying  properties stated in Lemma~\ref{lem:replace}. Let
\begin{align*}
\omega_{\widetilde q}(\rho)&:=\limsup_{\lambda\to\infty}\lambda\int_{C_\rho}\int_{Q_0}
\mathbf 1_{\left\{(x,y)\in C_\rho\times Q_0:\ x\neq y,\
\frac{|W_{\widetilde q}(x)-W_{\widetilde q}(y)|}{|x-y|^{\gamma+1}}\geq\lambda\right\}}
|x-y|^{\gamma-N}\,dy\,dx.
\end{align*}
From Lemma~\ref{lem:collar}, we infer that
\begin{align}\label{eq:repairsmall}
\lim_{\rho\to 0^+}\omega_{\widetilde q}(\rho)=0.
\end{align}
By Lemma~\ref{lem:replace}(iii), the definition of $\omega_{\widetilde q}(\rho)$, and \eqref{eq:profnear}, we find that
\begin{align}\label{eq:zj2}
\limsup_{j\to\infty}G_{\widetilde\lambda_j,1,\gamma}(\widehat v_j;Q_0)&\leq
\limsup_{j\to\infty}G_{\lambda_j,1,\gamma}(v_j;A_\rho)+\limsup_{j\to\infty}G_{\lambda_j,1,\gamma}(W_{\widetilde q};C_\rho)\notag\\
&\leq m_{1,\gamma}(W_{\widetilde q})+\varepsilon+\omega_{\widetilde q}(\rho).
\end{align}
Now, we claim that there exist sequences $\{\Lambda^{(m)}_j\}_{j\in\mathbb{N}}\subset (0,\infty)$ and
$\{V_j^{(m)}\}_{j\in\mathbb{N}}\subset L^1(Q_0)$ such that $\lim_{j\to\infty}\Lambda^{(m)}_j=\infty$,
$\lim_{j\to\infty}V_j^{(m)}=W_{q_m}$ in $L^1(Q_0)$, and
\begin{align}\label{eq:claim}
m_{1,\gamma}(W_{q_m})\leq\limsup_{j\to\infty}G_{\Lambda_j^{(m)},1,\gamma}(V_j^{(m)};Q_0)
\leq m_{1,\gamma}(W_{\widetilde q})+\varepsilon+\omega_{\widetilde q}(\rho).
\end{align}
If this claim holds, then, letting $\varepsilon,\rho\to 0^+$ and using \eqref{eq:repairsmall},
we conclude that \eqref{eq:zhongji} holds, which completes the proof of the present proposition.

Now, we prove the above claim. Note that the intervals $(a_0,a_1),(a_1,a_2),\ldots,(a_{m-1},a_m)$
are pairwise disjoint and cover $I_0$ except for the partition points.
Taking their Cartesian products in the $N$ coordinate directions
gives $m^N$ pairwise disjoint open cubes that cover $Q_0$ except for
their boundaries. More precisely, let
$\mathcal A_m:=\left\{0,\ldots,m-1\right\}^N.$
For any $\alpha=(\alpha_1,\ldots,\alpha_N)\in\mathcal A_m$, define
\begin{align*}
Q_\alpha:=\prod_{i=1}^N(a_{\alpha_i},a_{\alpha_i+1}).
\end{align*}
Then the family $\mathscr Q_m:=\left\{Q_\alpha:\alpha\in\mathcal A_m\right\}$
consists of pairwise disjoint open cubes and covers $Q_0$ except for
their boundaries.  Furthermore, for any $\alpha \in\mathcal{A}_m$, the center of $Q_\alpha$ is
\begin{align*}
c_\alpha&:=\left(\frac{a_{\alpha_1}+a_{\alpha_1+1}}{2},\ldots,\frac{a_{\alpha_N}+a_{\alpha_N+1}}{2}\right)\\
&=-\frac12(1,\ldots,1)+\frac{1}{m}\left(\alpha_1+\frac12,\ldots,\alpha_N+\frac12\right),
\end{align*}
and hence
\begin{align*}
Q_\alpha=c_\alpha+\frac{1}{m}Q_0.
\end{align*}
From the definitions of $W_{q_m}$ and $q_m$, it follows that, for any $\alpha\in \mathcal{A}_m$
with $k:=\alpha_N$ and $x\in Q_\alpha$,
\begin{align}\label{eq:juti}
W_{q_m}(x)&=q_m(x_N)=a_k+\frac{1}{m}\widetilde q\left(m(x_N-a_k)-\frac12\right)
=a_k+\frac{1}{m}\widetilde{q}\left(m\left(x_N-(c_\alpha)_N\right)\right)\notag\\
&=a_k+\frac{1}{m}\widetilde{q}\left([m(x-c_\alpha)]_N\right)
=a_k+\frac{1}{m}W_{\widetilde{q}}(m(x-c_\alpha)).
\end{align}
Let $j\in\mathbb{N}$. Replacing $W_{\widetilde q}$ in \eqref{eq:juti} by $\widehat{v}_j$,
let
\begin{align*}
V_j^{(m)}(x)
:=\begin{cases}
\displaystyle
a_{\alpha_N}+\frac{1}{m}
\widehat v_j\left(m(x-c_\alpha)\right),
&x\in Q_\alpha\text{ for some }\alpha\in\mathcal A_m,\\[2mm]
W_{q_m}(x),&\displaystyle x\in Q_0\setminus
\left(\bigcup_{\alpha\in\mathcal A_m}Q_\alpha\right).
 \end{cases}
\end{align*}
By the change of variables, we find that
\begin{align*}
\left\|V_j^{(m)}-W_{q_m}\right\|_{L^1(Q_0)}=\sum_{\alpha\in\mathcal A_m}
\int_{Q_\alpha}\left|V_j^{(m)}-W_{q_m}\right|\,dx=\frac{1}{m}
\left\|\widehat v_j-W_{\widetilde q}\right\|_{L^1(Q_0)}.
\end{align*}
Thus, $\lim_{j\to\infty}V_j^{(m)}=W_{q_m}$ in $L^1(Q_0)$. Define
$\Lambda_j^{(m)}:=m^\gamma\widetilde\lambda_j.$
From Lemma~\ref{lem:scaling}, we infer that, for any $\alpha\in\mathcal A_m$,
\begin{align*}
G_{\Lambda_j^{(m)},1,\gamma}\left(V_j^{(m)};Q_\alpha\right)=\frac{1}{m^N}
G_{\widetilde\lambda_j,1,\gamma}\left(\widehat v_j;Q_0\right).
\end{align*}
Consequently,
\begin{align}\label{eq:samecell}
\sum_{\alpha\in\mathcal A_m}G_{\Lambda_j^{(m)},1,\gamma}\left(V_j^{(m)};Q_\alpha\right)
=G_{\widetilde\lambda_j,1,\gamma}\left(\widehat v_j;Q_0\right).
\end{align}
Let
\begin{align*}
G_{\Lambda_j^{(m)},1,\gamma}^{\rm cross}\left(V_j^{(m)};\mathscr Q_m\right)
&:=\Lambda_j^{(m)}\sum_{\substack{\alpha,\beta\in\mathcal A_m\\\alpha\neq\beta}}
\int_{Q_\alpha}\int_{Q_\beta}\mathbf 1_{\left\{(x,y)\in Q_\alpha\times Q_\beta:\ x\neq y,\
\frac{\left|V_j^{(m)}(x)-V_j^{(m)}(y)\right|}{|x-y|^{\gamma+1}}
\geq\Lambda_j^{(m)}\right\}}|x-y|^{\gamma-N}\,dy\,dx.
\end{align*}
By \eqref{eq:samecell}, we conclude that
\begin{align}\label{eq:zj}
G_{\Lambda_j^{(m)},1,\gamma}\left(V_j^{(m)};Q_0\right)&\leq\sum_{\alpha\in\mathcal A_m}
G_{\Lambda_j^{(m)},1,\gamma}\left(V_j^{(m)};Q_\alpha\right)+ G_{\Lambda_j^{(m)},1,\gamma}^{\rm cross}
\left(V_j^{(m)};\mathscr Q_m\right)\notag\\
&=G_{\widetilde\lambda_j,1,\gamma}\left(\widehat v_j;Q_0\right)+ G_{\Lambda_j^{(m)},1,\gamma}^{\rm cross}
\left(V_j^{(m)};\mathscr Q_m\right).
\end{align}
For any $\alpha\in\mathcal A_m$, define
\begin{align*}
Q_{\alpha,\rho}:=\left\{x\in Q_\alpha:\operatorname{dist}(x,\partial Q_\alpha)>\frac{\rho}{m}\right\}.
\end{align*}
From Lemma~\ref{lem:replace}(i), we deduce that
\begin{align}\label{eq:tiledcollar}
V_j^{(m)}=W_{q_m}\quad\text{on }Q_\alpha\setminus Q_{\alpha,\rho}.
\end{align}
Since $\Lambda_j^{(m)}\to\infty$, it follows that
\begin{align*}
\left(\Lambda_j^{(m)}\right)^{-\frac{1}{\gamma+1}}<\frac{\rho}{m}
\end{align*}
for any sufficiently large $j$.  Let $j$ be any such index, let
$\alpha,\beta\in\mathcal A_m$ satisfy $\alpha\neq\beta$, and let
$(x,y)\in Q_\alpha\times Q_\beta$  satisfy
\begin{align*}
\left|V_j^{(m)}(x)-V_j^{(m)}(y)\right|\geq\Lambda_j^{(m)}|x-y|^{\gamma+1}.
\end{align*}
By the fact that $-\frac12\leq V_j^{(m)}\leq\frac12$, we find that
\begin{align*}
|x-y|\leq\left(\Lambda_j^{(m)}\right)^{-\frac{1}{\gamma+1}}<\frac{\rho}{m}.
\end{align*}
If $x\in Q_{\alpha,\rho}$, then $y\notin Q_\alpha$ and hence
\begin{align*}
|x-y|\geq\operatorname{dist}(x,\partial Q_\alpha)>\frac{\rho}{m},
\end{align*}
which is impossible.  Thus $x\in Q_\alpha\setminus Q_{\alpha,\rho}$. Applying
the same argument, we have $y\in Q_\beta\setminus Q_{\beta,\rho}$.
From this and \eqref{eq:tiledcollar}, we infer that
$V_j^{(m)}(x)=W_{q_m}(x)$ and $V_j^{(m)}(y)=W_{q_m}(y).$ Consequently,
\begin{align*}
&\left\{(x,y)\in Q_\alpha\times Q_\beta:\ x\neq y,\
\frac{\left|V_j^{(m)}(x)-V_j^{(m)}(y)\right|}{|x-y|^{\gamma+1}}
\geq\Lambda_j^{(m)}\right\}\\&\quad\subset
\left\{(x,y)\in Q_\alpha\times Q_\beta:\ x\neq y,\
\frac{\left|W_{q_m}(x)-W_{q_m}(y)\right|}{|x-y|^{\gamma+1}}
\geq\Lambda_j^{(m)}\right\},
\end{align*}
which implies that, for any sufficiently large $j$.
\begin{align}\label{eq:cross}
G_{\Lambda_j^{(m)},1,\gamma}^{\rm cross}\left(V_j^{(m)};\mathscr Q_m\right)\leq
G_{\Lambda_j^{(m)},1,\gamma}^{\rm cross}\left(W_{q_m};\mathscr Q_m\right).
\end{align}
Define $q_m^{\rm ext}:\mathbb R\to\mathbb R$ by setting, for any $t\in\mathbb R$,
\begin{align*}
q_m^{\rm ext}(t):=
\begin{cases}
-\dfrac{1}{2}, &t\in(-\infty,-\frac{1}{2}],\\[2mm]
q_m(t),&t\in(-\frac{1}{2},\frac{1}{2}),\\[2mm]
\dfrac{1}{2},&t\in[\frac{1}{2},\infty).
 \end{cases}
\end{align*}
Obviously, $q_m^{\rm ext}$ is nondecreasing and, for any $k\in\{0,\ldots,m\},$
\begin{align}\label{eq:qmnoatoms}
Dq_m^{\rm ext}\left(\left\{a_k\right\}\right)=0.
\end{align}
Choose $\chi\in C_{\rm c}^\infty(\mathbb R^N)$ such that $\chi=1$ on a open
neighborhood $U$ of $\overline{Q}_0$. For any $z\in\mathbb{R}^N,$ let
\begin{align*}
W_m^{\rm ext}(z):=\chi(z)q_m^{\rm ext}(z_N).
\end{align*}
Then $W_m^{\rm ext}\in BV(\mathbb R^N)\cap L^\infty(\mathbb R^N)$ and
$W_m^{\rm ext}=W_{q_m}$ on $Q_0$.
By an argument similar to that  used in \eqref{eq:dou2}, we conclude that, for any $i\in\{1,\ldots, N-1\}$,
\begin{align*}
(D_iW_m^{\rm ext})\llcorner U=0\ \ \text{and}\ \ (D_NW_m^{\rm ext})\llcorner U
=\left(\mathcal L^{N-1}\otimes Dq_m^{\rm ext}\right)\llcorner U.
\end{align*}
This implies that
\begin{align*}
|DW_m^{\rm ext}|\llcorner U=\left(\mathcal L^{N-1}\otimes|Dq_m^{\rm ext}|\right)\llcorner U.
\end{align*}
and hence, for any Borel sets $A'\subset\mathbb R^{N-1}$ and
$B\subset\mathbb R$ satisfying $A'\times B\subset U$,
\begin{align*}
|DW_m^{\rm ext}|(A'\times B)=\mathcal L^{N-1}(A')|Dq_m^{\rm ext}|(B).
\end{align*}
From this and \eqref{eq:qmnoatoms}, we deduce that
\begin{align*}
|DW_m^{\rm ext}|(S_{\mathscr Q_m})=0.
\end{align*}
Thus, all  assumptions of Lemma~\ref{lem:skeleton} are satisfied. Using this and
applying Lemma~\ref{lem:skeleton}  with $U:=W_{q_{m}}$
and $\mathscr Q:=\mathscr Q_m$, we find that
\begin{align*}
\lim_{\lambda\to\infty}G_{\lambda,1,\gamma}^{\rm cross}(W_{q_m};\mathscr Q_m)=0
\end{align*}
This, combined with \eqref{eq:cross}, implies that
\begin{align*}
\lim_{j\to\infty} G_{\Lambda_j^{(m)},1,\gamma}^{\rm cross}\left(V_j^{(m)};\mathscr Q_m\right)
=\lim_{j\to\infty}G_{\Lambda_j^{(m)},1,\gamma}^{\rm cross}\left(W_{q_m};\mathscr Q_m\right)=0.
\end{align*}
From this, the definition of $m_{1,\gamma}(W_{q_m})$, \eqref{eq:zj}, and \eqref{eq:zj2}, we infer that
\begin{align*}
m_{1,\gamma}\left(W_{q_m}\right)&\leq\limsup_{j\to\infty}G_{\Lambda_j^{(m)},1,\gamma}\left(V_j^{(m)};Q_0\right)\\
&\leq\limsup_{j\to\infty}\left[G_{\widetilde\lambda_j,1,\gamma}\left(\widehat v_j;Q_0\right)+ G_{\Lambda_j^{(m)},1,\gamma}^{\rm cross}
\left(V_j^{(m)};\mathscr Q_m\right)\right]\leq m_{1,\gamma}\left(W_{\widetilde q}\right)+\varepsilon
+\omega_{\widetilde q}(\rho).
\end{align*}
This shows the claim \eqref{eq:claim}, which completes the proof of Proposition \ref{prop:profile}.
\end{proof}

\subsection{Proof of Theorem \ref{thm:oneliminf}\label{proofone}}

In this subsection, we prove Theorem \ref{thm:oneliminf} by using Proposition \ref{prop:profile}
and Lemma \ref{lem:tangent}.

\begin{proof}[Proof of Theorem \ref{thm:oneliminf}]
Let $L:=\liminf_{j\to\infty}G_{\lambda_j,1,\gamma}(u_j;\Omega).$
If $L=\infty$, then \eqref{eq:oneliminf} holds automatically. Assume that $L<\infty$.
After passing to a subsequence and relabeling, we may assume that $G_{\lambda_j,1,\gamma}(u_j;\Omega)\to L.$
By Lemma~\ref{lem:gp}, we conclude that $u\in BV(\Omega)$.
If $|Du|(\Omega)=0$, then \eqref{eq:oneliminf} follows from the
nonnegativity of $G_{\lambda_j,1,\gamma}$. Hence we assume that $|Du|(\Omega)>0$.

Let $A_1\subset\Omega$ be a Borel set with
$|Du|(\Omega\setminus A_1)=0$ such that Lemma~\ref{lem:tangent}
holds at every point of $A_1$. Fix $x_0\in A_1$ and let $k\in\mathbb N$. For any $z\in Q_0$, let
\begin{align*}
Q_k:=Q_{r_k}^{R_{x_0}}(x_0),\ \ A_k:=\frac{|Du|(Q_k)}{r_k^{N-1}},
\ \ b_k:=\fint_{Q_k}u\,dy,
\end{align*}
and
\begin{align*}
v_k(z):=\frac{u(x_0+r_kR_{x_0}z)-b_k}{A_k}.
\end{align*}
By Lemma~\ref{lem:tangent}(i), we conclude that $Q_k\subset\Omega$ and
$|Du|(Q_k)\in(0,\infty)$.
For any $z\in Q_0$,  define
\begin{align*}
v_{j,k}(z):=\frac{u_j(x_0+r_kR_{x_0}z)-b_k}{A_k}\ \ \text{and}\ \
\Lambda_{j,k}:=\frac{\lambda_jr_k^{\gamma+1}}{A_k}.
\end{align*}
Then $\lim_{j\to\infty}v_{j,k}=v_k$ in $L^1(Q_0)$ and
$\lim_{j\to\infty}\Lambda_{j,k}=\infty$. Using Lemma~\ref{lem:scaling} with $p:=1$,
$a:=\frac{A_k}{r_k}$, $b:=b_k$, and $w:=v_{j,k}$, we find that
\begin{align}\label{eq:localscale}
G_{\lambda_j,1,\gamma}(u_j;Q_k)=\frac{A_k}{r_k}r_k^NG_{\frac{\lambda_jr_k^{\gamma+1}}{A_k},1,\gamma}
(v_{j,k};Q_0)=|Du|(Q_k)G_{\Lambda_{j,k},1,\gamma}(v_{j,k};Q_0).
\end{align}
Using \eqref{eq:localscale} and the definition of $m_{1,\gamma}$, we conclude that
\begin{align}\label{eq:localone}
\liminf_{j\to\infty}G_{\lambda_j,1,\gamma}(u_j;Q_k)
&= |Du|(Q_k)\liminf_{j\to\infty}G_{\Lambda_{j,k},1,\gamma}(v_{j,k};Q_0)\notag\\
&\geq |Du|(Q_k)m_{1,\gamma}(v_k).
\end{align}
Furthermore, by Lemma~\ref{lem:tangent}, we conclude that
$v_k\to W_{q_{x_0}}$ in $L^1(Q_0)$ and $|Dq_{x_0}|(I_0)=1$.
This, together with Lemma~\ref{lem:mplsc} and Proposition \ref{prop:profile}, implies that
\begin{align}
\liminf_{k\to\infty}m_{1,\gamma}(v_k)\geq m_{1,\gamma}(W_{q_{x_0}})
\geq C_{N,1,\gamma}^{\mathrm{cell}}|Dq_{x_0}|(I_0)=C_{N,1,\gamma}^{\mathrm{cell}}.\label{eq:mtangent}
\end{align}
Fix $\varepsilon\in(0,\frac{C_{N,1,\gamma}^{\mathrm{cell}}}{2})$.
Using
\eqref{eq:mtangent}, we conclude  that there exists
$k_\varepsilon(x_0)\in\mathbb{N} $ such that, for any
$k\in\mathbb N\cap(k_\varepsilon(x_0),\infty)$,
\begin{align*}
m_{1,\gamma}(v_k)
\geq C_{N,1,\gamma}^{\mathrm{cell}}-\varepsilon.
\end{align*}
From this
and \eqref{eq:localone}, we deduce that, for any
$k\in\mathbb N\cap(k_\varepsilon(x_0),\infty)$,
\begin{align}\label{eq:onelocal}
\liminf_{j\to\infty}G_{\lambda_j,1,\gamma}(u_j;Q_k)
\geq(C_{N,1,\gamma}^{\mathrm{cell}}-\varepsilon)|Du|(Q_k).
\end{align}
Now, define
\begin{align*}
V_{\varepsilon}:=\left\{Q_{r_k}^{R_x}(x):x\in A_1,\ k>k_\varepsilon(x)\right\}.
\end{align*}
Then, applying Lemma~\ref{lem:morse} to $(|Du|, \Omega, V_\varepsilon)$,
we further find that there exists a countable subfamily
$\{Q_i\}_{i\in\mathcal I}\subset V_\varepsilon$ of pairwise disjoint cubes such that
\begin{align*}
|Du|\left(A_1\setminus\left(\bigcup_{i\in\mathcal I}Q_i\right)\right)=0.
\end{align*}
This, combined with the fact that $|Du|(\Omega\setminus A_1)=0$, implies
\begin{align*}
|Du|\left(\Omega\setminus\left(\bigcup_{i\in\mathcal I}Q_i\right)\right)=0.
\end{align*}
Since $|Du|(\Omega)<\infty$, we can choose a finite subfamily of $\{Q_i\}_{i\in\mathcal I}$,
relabeled as $Q_1,\ldots,Q_M$, such that
\begin{align*}
|Du|\left(\Omega\setminus\left(\bigcup_{i=1}^M Q_i\right)\right)<\varepsilon.
\end{align*}
Using this and \eqref{eq:onelocal}, we obtain that
\begin{align*}
\liminf_{j\to\infty}G_{\lambda_j,1,\gamma}(u_j;\Omega)
&\geq\liminf_{j\to\infty}\sum_{i=1}^M G_{\lambda_j,1,\gamma}(u_j;Q_i)\\
&\geq(C_{N,1,\gamma}^{\mathrm{cell}}-\varepsilon)\sum_{i=1}^M|Du|(Q_i)
>(C_{N,1,\gamma}^{\mathrm{cell}}-\varepsilon)(|Du|(\Omega)-\varepsilon).
\end{align*}
Letting $\varepsilon\to 0$, we find that \eqref{eq:oneliminf} holds.
This completes the
proof of Theorem \ref{thm:oneliminf}.
\end{proof}

\smallskip
\noindent\textbf{Acknowledgements}\quad
The authors acknowledge the use of AI tools during the exploratory stage of this project.
All mathematical arguments and proofs in the final manuscript were checked
and written by the authors.

\bigskip

\noindent
Xiaosheng Lin

\smallskip

\noindent
School of Mathematical Sciences, Jimei University,
Xiamen 361005, The People's Republic of China

\smallskip

\noindent {\it E-mail}: \texttt{xslin@jmu.edu.cn}

\bigskip

\noindent Dachun Yang, Wen Yuan and Yangyang Zhang

\smallskip

\noindent Laboratory of Mathematics and Complex Systems
(Ministry of Education of China),
School of Mathematical Sciences, Institute for Advanced Study,
Beijing Normal University,
Beijing 100875, The People's Republic of China

\smallskip

\noindent{\it E-mails:} \texttt{dcyang@bnu.edu.cn} (D. Yang)

\noindent\phantom{{\it E-mails:}} \texttt{wenyuan@bnu.edu.cn} (W. Yuan)

\noindent\phantom{{\it E-mails:}} \texttt{yangyzhang@bnu.edu.cn} (Y. Zhang)

\bigskip

\noindent Sibei Yang

\medskip

\noindent School of Mathematics and Statistics, Lanzhou University, Lanzhou 730000, The People's Republic of China

\smallskip

\noindent{\it E-mail:} \texttt{yangsb@lzu.edu.cn}

\end{document}